%% file: SIAM_MS_WHHMM.tex
\documentclass[hidelinks,onefignum,onetabnum]{siamart251216}

\usepackage{amsfonts}
\usepackage{graphicx}
\usepackage{epstopdf}
\usepackage{algorithmic}
\ifpdf
  \DeclareGraphicsExtensions{.eps,.pdf,.png,.jpg}
\else
  \DeclareGraphicsExtensions{.eps}
\fi

\usepackage{amsfonts}
\usepackage{amssymb}
\usepackage{enumitem}

\newsiamremark{remark}{Remark}
\newsiamremark{hypothesis}{Hypothesis}
\crefname{hypothesis}{Hypothesis}{Hypotheses}
\newsiamthm{claim}{Claim}
\newsiamremark{fact}{Fact}
\crefname{fact}{Fact}{Facts}
\newsiamthm{prop}{Proposition}

\headers{WaveHoltz HMM}{A. Rotem, O. Runborg, and D. Appel\"{o}}

\title{The WaveHoltz Heterogeneous Multiscale Method.\thanks{Submitted to the editors \date.
\funding{DA is supported by the U.S. Department of Energy, Office of Science, Advanced Scientific Computing Research (ASCR), under Award Number DE-SC0025424. This material is based upon work supported, in part, by the National Science Foundation under Grant No. DMS-2436319 and Virginia Tech. This material is based upon work supported by the National Science Foundation under Grant No. DMS-2424139 while the third author was in residence at the Simons Laufer Mathematical Sciences Institute in Berkeley, California, during the Fall 2025 semester.}}}

\author{Amit Rotem\thanks{Department of Mathematics, Virginia Tech
  (\email{arotem@vt.edu}).}
  \and Olof Runborg\thanks{Department of Mathematics, KTH Royal
  Institute of Technology (\email{olofr@kth.se}).}
  \and Daniel Appel\"{o}\thanks{Department of Mathematics, Virginia
Tech (\email{appelo@vt.edu}).}}

\usepackage{amsopn}
\DeclareMathOperator{\diag}{diag}

\newcommand{\N}{\mathbb{N}}

\newcommand{\R}{\mathbb{R}}

\renewcommand{\O}{\mathcal{O}}

\renewcommand{\vec}[1]{\mathbf{#1}}

\newcommand{\grad}{\nabla}

\ifpdf
\hypersetup{
  pdftitle={The WaveHoltz Heterogeneous Multiscale Method},
  pdfauthor={A. Rotem, O. Runborg, and D. Appel\'{o}}
}
\fi

\begin{document}

\maketitle

\begin{abstract}
  We consider the numerical solution of
  the wave equation in
  materials with rapidly varying coefficients, and time harmonic
  sources. For these problems, direct discretization is prohibitively
  costly, and instead multiscale methods are used. There are several
  multiscale methods that directly discretize in the frequency
  domain. In this work we instead start in the time-domain and
  combine a finite difference Heterogeneous Multiscale Method (HMM)
  for the wave equation with the WaveHoltz method. Each WaveHoltz
  iteration marches the  wave equation towards the time-periodic
  Helmholtz solution. The advantages of the WaveHoltz method relative
  to traditional Helmholtz solvers carry over directly to the
  multiscale problems considered here. Since, in addition, the
  time-domain solver does not artificially impose boundary conditions
  on the micro-scale problems, no boundary errors from the
  micro-scale problems are present in the homogenized frequency domain solution.
\end{abstract}



\section{Introduction}

This paper considers the Helmholtz equation,
\begin{gather} \label{eq:helmholtz}
  \nabla \cdot (a_\varepsilon(x) \nabla u) + \omega^2 u = b(x),
  \qquad x\in\Omega,
\end{gather}
with Dirichlet or Neumann boundary conditions. Here the computational
domain $\Omega\subset\R^d$ ($d = 1$ or 2) is a smooth, bounded, and
simply connected domain, $b\in L^2(\Omega)$
is independent of
$\varepsilon$ and the Helmholtz frequency
 $\omega$ is assumed to be non-resonant.
The scalar valued function $a_\varepsilon(x)$ is uniformly bounded
and positive. 
We are interested in the case where
$a_\varepsilon$ is rapidly varying, i.e.\mbox{} its smallest scale
$\varepsilon$ is much smaller than the longest wavelengths  $\propto
\| \sqrt{a_\varepsilon} \|_{L^{\infty}}/ \omega$ of the
solution. Our goal is to compute approximations to the solution $u$
for problems with scale separation. 
By scale separation we
mean that the solution $u(x)$ can be expanded in powers of $\varepsilon$ as
\[
  u(x) = \bar{u}(x) + \varepsilon u_1(x,x / \varepsilon) + \O(\varepsilon^2),
\]
where, the zeroth order term $\bar{u}$ does not depend on
$\varepsilon$. In the particular case when $a_\varepsilon(x) = a(x,
x/\varepsilon)$ is periodic in the fast variable $x/\varepsilon$, the
zeroth order term $\bar{u}$ can be shown \cite{Marchenko2005-ov,
Lions-Periodic-Structures, Allaire-Two-Scale-Convergence} to satisfy
a homogenized equation
\begin{gather}
  \nabla\cdot(\bar{a}(x)\nabla\bar{u}) + \omega^2 \bar{u} = b(x),
  \qquad x\in \Omega. \label{eq:hom-helmholtz}
\end{gather}
Note that $\bar{a}$ {\em does not} depend on $\varepsilon$ and thus
$\bar{u}$ need only be resolved at the length scales of the
homogenized coefficient $\bar{a}$, the wavelengths associated with
the frequency $\omega$, and the scales of the forcing $b$. When these
are large compared to $\varepsilon$ and when $\varepsilon$ is small
enough for the  homogenized solution $\bar{u}$ to be a good
approximation of $u$, homogenization provides an accurate
approximation to $u$ at a low computational cost.

However, solving or discretizing \eqref{eq:hom-helmholtz} requires explicit
knowledge of $\bar{a}$. When $a_\varepsilon$ is periodic and depends
only on the fast scale, i.e. $a_\varepsilon(x) = a(x/\varepsilon)$,
the homogenized coefficient $\bar{a}$ can be computed by solving the
so called cell problem, an associated elliptic problem with periodic
boundary conditions, \cite{Lions-Periodic-Structures}. When
$a_\varepsilon$ is $\varepsilon$-periodic but also has a slowly
varying component, ($a_\varepsilon(x) = a(x, x/\varepsilon) = a(x,
  (x+x_{\rm per)}/\varepsilon)$, the cell problem will depend on the
  slow variable $x$ adding a dimension of complexity to the problem.
  A broad goal of numerical homogenization is to relax the condition
  of $\varepsilon$-periodicity of $a_\varepsilon(x)$.

  \subsection{Literature Review}

   Multiscale methods have been extensively studied in their
  application to partial differential equations. For classical
  elliptic problems with rapidly varying coefficients, 
  the heterogeneous multiscale method (HMM), the multiscale
  finite element method and its generalizations (MsFEM and GMsFEM,
  respectively) and local orthogonal decomposition (LOD) have been
  particularly successful \cite{hmm-review-engquist,Efendiev2009-fd,
  Altmann_Henning_Peterseim_2021}.
These approaches have also been used for wave equations and 
to a lesser extent, to the Helmholtz equation and time harmonic Maxwell equations.
The numerical homogenization of the
  Helmholtz equation introduces additional difficulties as
  the equation is no longer coercive, and may not be well posed for all
frequencies.

  MsFEM departs from classic finite element methods by constructing
  basis functions for representing the solution on a coarse mesh from
  the solutions of local problems on each element. Since MsFEM
  directly discretizes the multiscale PDE, it does not rely on
  assumptions of scale separation or periodicity in the coefficients.
  However, the imposition of artificial boundary conditions on the
  local problems can introduce errors. To compensate for these
  errors, oversampling techniques are used \cite{Efendiev2009-fd}. A
  popular extension of MsFEM is GMsFEM (Generalized MsFEM). In GMsFEM
  several multiscale basis functions are chosen by solving local
  eigenvalue problems on each element. Some of these methods solve
  the local eigenvalue problem using snapshots and the proper
  orthogonal decomposition, whereas others directly discretize the
  coarse element using fine scale basis functions. The choice of
  eigenvalue problem, snapshots, and oversampling region all play a
  critical role in the quality of multiscale basis for approximating
  the coarse scale solution. GMsFEM often provide a substantial
  improvement in accuracy over classic MsFEM \cite{Chung2016-bk,Chung2023-lv,
  Chung2018-zw,  Efendiev2009-fd}.
MsFEM-type methods have also been developed for wave propagation problems; examples include works in the time domain \cite{ChungEfendievLeung2014,GaoFuChung2018} and in the frequency domain \cite{Chaumont-multiscale-helmholtz,Chen-exponential-multiscale}.

  LOD methods approximate multiscale solutions by introducing an
  interpolation operator that connects fine and coarse approximation
  spaces. From the kernel of this operator, a localized multiscale
  space is defined, and basis functions are computed by solving local
  problems. The global solution of the PDE is then approximated using
  these localized basis functions. The accuracy of this approach
  depends on the oversampling region around each coarse element, and
  the error decays exponentially with respect to the radius of the
  oversampling region. LOD has proven particularly effective for
  problems with highly oscillatory or rough coefficients
\cite{Altmann_Henning_Peterseim_2021,Malqvist2020-ej}. Both MsFEM
  and LOD ultimately require resolution of the entire domain, albeit
  through decoupled local computations.
Extensions of LOD to time-domain wave equations can be found, for example, in \cite{Abdulle2016-hm,KrumbiegelMaier,MaierPetersheim}, while LOD-type methods for Helmholtz problems are considered in \cite{BrownGallistlPeterseim2017,HauckPeterseim2022}.

  HMM specializes to problems in homogenization by implicitly
  approximating the homogenized operator, rather than the solution
  directly, through localized micro-scale simulations \cite{hmm-review-engquist}. 
The main idea is to assume the form of the effective equation is
  known and discretize it directly.
For coercive
  elliptic equations, several variants of HMM have been developed,
  differing primarily in the choice of macroscale model and
numerical methods for the micro and macro scales
 \cite{FD-HMM, FEM-HMM-E,ABDULLE20091081}.
In all cases, these methods are
  characterized by a coupling of local problems to a global problem.
  The local problems of HMM are usually very ``sparse'' in the
  domain, that is, the solution is only fully resolved on a small
  fraction of the domain. This is possible because of the scale
  separation ansatz of HMMs. Like MsFEM and LOD, HMM imposes an
  artificial truncation of the local problems which introduces
  additional approximation errors. Unlike MsFEM and LOD, HMM is
  generally not well suited to problems without scale separation.
HMM methods for wave equations have  been developed for both short and long time scales, including approaches that incorporate higher-order dispersive terms \cite{wave-hmm, FE-HMM, ArjmandRunborg2014, AbdulleGroteStohrer2014}. Notable approaches for time-harmonic equations can be found in \cite{CIARLET2014755, Henning-maxwell, hmm-helmholtz-Ohlberger}.

Finally, we note that multiscale wave equations have also been treated using
several other frameworks, including equation-free methods
\cite{ArjmandKreiss2018} and harmonic coordinates
\cite{OwhadiZhang}.

  The WaveHoltz HMM (WHMM) method that we propose 
  combines an HMM methods for the wave equation
  with the iterative method WaveHoltz \cite{WaveHoltz} for  solving Helmholtz equations.
It  
  can be categorized as a
  numerical homogenization method for \eqref{eq:helmholtz}. The WHMM
  is a fixed point iteration which converges to $\bar{u}$, the
  solution to \eqref{eq:hom-helmholtz}. This fixed point iteration
  can be accelerated by Krylov methods. Further, it inherits several
  desirable properties from WaveHoltz and heterogeneous multiscale
  methods for the wave equation making it an effective iterative
  solver for approximating the solution to \eqref{eq:hom-helmholtz}.
  In particular, the finite speed of propagation in the wave equation
  eliminates errors associated with non-physical boundary conditions
  on the fine scale seen in other methods for elliptic homogenization
  problems. Additionally, the HMM in \cite{wave-hmm} introduces the
  use of local averaging kernels over the micro-scale problem which
  exponentially improve the approximation of the homogenized
  coefficient for smooth problems. Further, when the boundary
  conditions on $\partial\Omega$ are energy conserving (Dirichlet,
  Neumann, periodic), the linear system associated with the fixed
  point iteration is symmetric positive definite, thus the fixed
  point iteration can be accelerated by the conjugate gradient method.

  We outline the paper as follows: in Section
  \ref{sec:whmm-continuous} we describe the WaveHoltz HMM beginning
  with the Heterogeneous Multiscale Method for the wave equation,
  including both its continuous formulation and finite difference
  discretization in Section \ref{sec:wave-hmm-continuous}, followed
  by the continuous WaveHoltz iteration in Section
  \ref{sec:waveholtz-continuous}. We then combine the discretized
  wave equation with the discrete WaveHoltz iteration to obtain the
  WaveHoltz Heterogeneous Multiscale Method in Section
  \ref{sec:discrete-method}, and describe how the homogenized
  coefficient can be approximated offline in Section
  \ref{sec:offline-coefficient}. In Section \ref{sec:1d-analysis} we
  analyze the convergence of the method in a single dimension and a
  periodic coefficient $a_\varepsilon(x)$. Finally, in Sections
  \ref{sec:experiments1D} and \ref{sec:experiments2D} we provide
  numerical examples in one and two dimensions, respectively, for
  problems where the exact homogenized coefficient $\bar{a}$ is known.

  \section{The WaveHoltz Heterogeneous Multiscale
  Method}\label{sec:whmm-continuous}

  The aim of the WaveHoltz Heterogeneous Multiscale Method is to solve
  \eqref{eq:hom-helmholtz} at $\varepsilon$-independent cost, without a
  priori knowledge of $\bar{a}$ and without resolving $a_\varepsilon$
on the entire domain. WHMM has two core components: 1.) A HMM for the
wave equation (here we use \cite{wave-hmm}, but see also
\cite{FE-HMM}), and 2.) the WaveHoltz iteration \cite{WaveHoltz}.

\subsection{The Heterogenous Multiscale Method for the Wave Equation}
\label{sec:wave-hmm-continuous}
We first describe how HMM works for the time-dependent wave equation
of the following form
\begin{gather}\label{eq:wave}
w_{tt} = \nabla\cdot(a_\varepsilon(x)\nabla w) +f(t,x), \\
w(x,0) = v_0(x), \quad w_t(x,0) = v_1(x), \notag
\end{gather}
with boundary conditions suitable for the problem at hand. The
coefficient $a_\varepsilon$ is assumed to be rapidly varying with
smallest scale $\varepsilon$, as discussed in the introduction, and
$f$ is a source function independent of $\varepsilon$. An HMM method for the wave
equation \eqref{eq:wave} finds an approximation of the homogenized
solution $\bar{w}$, that is assumed to satisfy a coarse scale wave
equation of the form
\begin{gather}\label{eq:eff-wave-eq}
\bar{w}_{tt} = \nabla\cdot F(t,x) +f(t,x), \\
w(x,0) = v_0(x), \quad w_t(x,0) = v_1(x), \notag
\end{gather}
where $F$ is an unknown macroscopic flux function, the form of which
depends on how the coefficient  $a_\varepsilon$ is homogenized. For
instance, if $a_\varepsilon$ is $\varepsilon$-periodic, then
$F=\bar{a}\nabla w$,
where $\bar{a}$ is the same homogenized coefficient that appears in
\eqref{eq:hom-helmholtz}, see \cite{wave-hmm}. In WHMM, 
$f(t,x)=-b(x)\cos(\omega t)$ and
the initial
data $v_0$, $v_1$ are specified by the WaveHoltz iteration; see Section \ref{sec:waveholtz-continuous}.

The solution $\bar{w}$ to \eqref{eq:eff-wave-eq} is computed on a
coarse finite difference mesh, which is coupled to localized,
micro-scale, solutions of \eqref{eq:wave} on fully resolved grids
with a fixed number of grid points independent of $\varepsilon$. The
solution $\bar{w}$ is interpolated on the macro-scale and used as
initial conditions to the local micro-scale solver. The flux
$a_\varepsilon\grad w$ of the localized solution is averaged over
short time and length scales to approximate the macroscopic flux $F$
in \eqref{eq:eff-wave-eq}.

Precise specification of a HMM for the wave equation requires
specification of the discretization used. We will provide a semi-discretization below
in one dimension. In higher dimensions we prefer to use a staggered
grid finite difference method, such as described in
\cite{aniso-diffusion, mimetic-fd, hom-fd}, which ensures symmetry of
the finite difference operator when the homogenized coefficient is
anisotropic. We thus consider the one dimensional wave equation for
$w = w(x,t)$,
\begin{gather}
w_{tt} = \left( a_\varepsilon(x) w_x  \right)_x + f(t,x), \qquad
x\in(0,L), \; t>0, \label{eq:wave-1d} \\
w(0,t) = w(L, t) = 0, \notag \\
w(x,0) = v_0(x), \qquad w_t(x,0)= v_1(x). \notag
\end{gather}
While other boundary conditions are of practical interest, for
brevity we only consider homogeneous Dirichlet conditions here as
extensions to Neumann and Robin boundary conditions are well
documented for finite difference methods, see for example \cite{LeVeque2007-oq}.

\subsubsection{Macro-scale Solver} \label{sec:macro}
Let the coarse finite difference grid with $M_x+2$ points be $X_j =
Hj = \frac{L}{M_x+1}j,$ \mbox{$j=0,\dots,M_x+1$}. Let $\bar{w}_j(t)$
be a grid function approximating $\bar{w}(X_j, t)$, the solution to
\eqref{eq:eff-wave-eq}. Then the semi-discretization in space is:
\begin{subequations} \label{eq:wave-semi-discrete}
\begin{gather}
  \frac{d^2 \bar{w}_j}{d t^2} = \frac{F_{j+\frac{1}{2}}(t) -
  F_{j-\frac{1}{2}}(t)}{H} + f(t,X_j), \qquad j = 1, \dots, M_x, \; t > 0, \\
  \bar{w}_0(t) = \bar{w}_{M_x+1}(t) = 0, \qquad t>0, \\
  \bar{w}_j(0) = v_0(X_j), \qquad \frac{d \bar{w}_j}{d t}\bigg|_{t=0} = 0.
\end{gather}
\end{subequations}
Here $F_{j+\frac{1}{2}}(t)$ is an approximation of
$F(t,X_{j+\frac{1}{2}})=F(t,X_j + H/2)$ and is the output of a fine
grid micro simulation centered at $X_{j+\frac{1}{2}}$ which we now describe.

\subsubsection{Micro-scale Solver} \label{sec:micro-solver}

\begin{figure}[]
\centering
\includegraphics[width=0.5\textwidth]{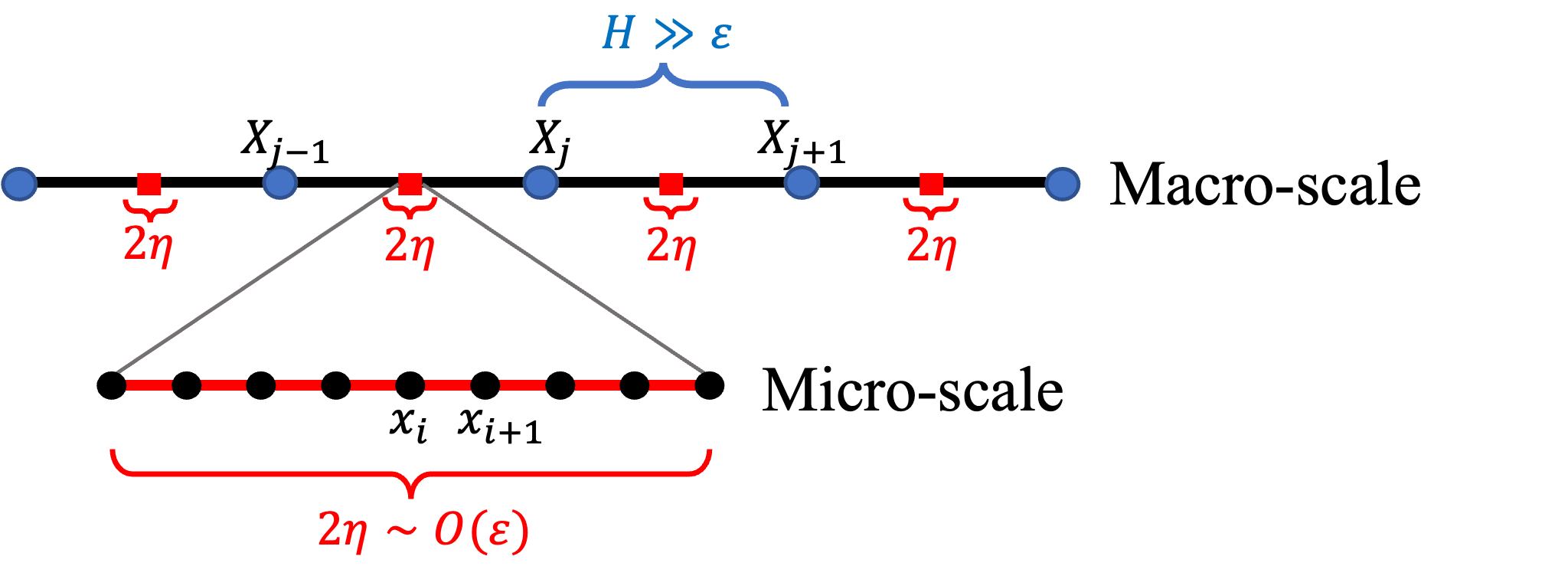}
\caption{The macro-scale mesh and micro-scale mesh.}
\label{fig:macro-micro}
\end{figure}

The micro-scale solver is a local discretization of
\eqref{eq:wave-1d} centered at $X_{j+\frac{1}{2}}$. Let $\eta, \tau
\sim O(\varepsilon)$. On the domain $(x,t)\in [-\eta, \eta]\times
[-\tau, \tau]$ we solve the original variable coefficient wave
equation (without forcing)
\begin{gather}
w_{tt} = \bigl(a_\varepsilon(X_{j+\frac{1}{2}} + x) w_x\bigr)_x,
\label{eq:fine-scale-wave} \\
w(x,0) = \bar{w}_{j+\frac{1}{2}} + P_{j+\frac{1}{2}} \cdot x, \qquad
w_t(x,0) = 0. \notag
\end{gather}
Here, at time $t = 0$, the microscale solution is taken to be the
linear interpolant of $\bar{w}_j$ and $\bar{w}_{j+1}$
\begin{equation}\label{eq:Pjdef}
P_{j+\frac{1}{2}} := \frac{\bar{w}_{j+1} - \bar{w}_{j}}{H} \approx
\bar{w}_x (X_{j+\frac{1}{2}}).
\end{equation}
We take the parameters $\eta, \tau \ll H$ so that only a small part
of the domain is included in each micro solve, but sufficiently
larger than (an $O(10)$ multiple of) $\varepsilon$ to effectively
average away local oscillations of the solution. Note the time-interval $[-\tau, \tau]$ includes negative time. This poses no difficulties because the solution $w$ to \eqref{eq:fine-scale-wave} is symmetric in time. The coarse and fine
grids are visualized in Figure \ref{fig:macro-micro}. The solution
varies smoothly on each micro-scale domain, so it can be resolved
with a modest number of discretization points.

Since the homogenized solution propagates at a wave speed not
exceeding $\|\sqrt{a_\varepsilon}\|_{L^\infty}$, no information
from neighboring grid-points $X_{j-1/2}$ and $X_{j+3/2}$ propagates
into the small scale domain over the time interval $[-\tau,\tau]$
provided that $\eta$ and $\tau$ are selected to be sufficiently
small. Hence, a logical choice of boundary conditions are open
(outflow) boundary conditions which can be well approximated as a
Robin condition, or a PML. At a slight increase in computational
cost, the added algorithmic complexity introduced by a PML can be
avoided by simply extending the computational domain to the domain of
dependence $[-\hat{\eta}, \hat{\eta}]$ where
$\hat{\eta} = \eta + \tau \cdot \|\sqrt{a_\varepsilon}\|_{L^\infty}.$

The solution $w(x,t)$ can be approximated with any discretization
method. Here we use the same finite difference discretization. Define
the fine scale finite difference grid on the the interval
$[-\hat{\eta}, \hat{\eta}]$ with $m_x+2$ points,
$$
x_i = -\hat{\eta} + hi = -\hat\eta + \tfrac{2\hat\eta}{m_x+1}i,
\qquad i=0, \dots, m_x+1.
$$
Let $w_i^\ell$ be a grid function approximating $w(x_i, t^\ell)$. The
approximation is advanced with time step $k=\tau/m_t$ (for a total of
$m_t$ time-steps) according to the difference equation for
$i\in\{1,\dots,m_x\}$ and $\ell\in\{2,\dots,m_t\}$,
\begin{subequations}
\begin{gather*}
  \frac{w_i^{\ell+1} - 2 w_i^\ell + w_i^{\ell-1}}{k^2} =
  \frac{a_{i+\frac{1}{2}}(w_{i+1}^\ell - w_i^\ell) -
  a_{i-\frac{1}{2}}(w_i^\ell - w_{i-1}^\ell)}{h^2}, \\
  w_0^\ell = w_{m_x+1}^\ell = 0, \\
  w^0_i = \bar{w}_{j+\frac{1}{2}} + P_{j+\frac{1}{2}} x_i,
  \quad
  w^1_i=w_i^0 + \tfrac{1}{2}k^2 \frac{a_{i+\frac{1}{2}}(w_{i+1}^\ell
  - w_i^\ell) - a_{i-\frac{1}{2}}(w_i^\ell - w_{i-1}^\ell)}{h^2},
\end{gather*}
\end{subequations}
Here $a_{i+\frac{1}{2}} = a_\varepsilon(X_{j+\frac{1}{2}} + x_{i+\frac{1}{2}}).$

It is a standard result of classical homogenization theory
\cite{Lions-Periodic-Structures, Allaire-Two-Scale-Convergence,
Marchenko2005-ov} that the average value of the flux $a_\varepsilon
\grad w$ over a period $\varepsilon$ converges to the homogenized
flux $\bar{a} \grad\bar{w}$ in the periodic setting for elliptic
problems. Motivated by these results, the flux $F_{j+\frac{1}{2}}$ is
therefore approximated by the weighted average in both space and time
of $a_\varepsilon \grad w$,
\begin{equation} \label{eq:flux-averaging}
F_{j+\frac{1}{2}} = \frac{1}{\eta\tau}
\int_{-\tau}^{\tau}\int_{-\eta}^{\eta}
\mathcal{K}(x/\eta)\mathcal{K}(t/\tau)
a_\varepsilon(X_{j+\frac{1}{2}}+x) \frac{\partial w}{\partial x} \, dx\, dt.
\end{equation}
Here, $\mathcal{K}\in \mathbb{K}^{p,q}$ is a local averaging kernel.
A function $\mathcal{K}$ is said to belong to $\mathbb{K}^{p,q}$ if
it has $q$ continuous and compactly supported derivatives in
$[-1,1]$, and has $p$ vanishing moments. The integral in
\eqref{eq:flux-averaging} is approximated by the trapezoidal rule on
the grid $(x_i,t^n)$ which is high order accurate when $q$ is large.
We note finally, that by linearity, we can write
$F_{j+\frac{1}{2}}(t)=\alpha_{j+\frac{1}{2}}P_{j+\frac{1}{2}}(t)$,
where $\alpha_{j+\frac{1}{2}}$ does not depend on the macroscale grid
function $\bar{w}_j$, nor on $t$, and $P_{j,+\frac{1}{2}}$. Moreover,
in \cite{wave-hmm} it was shown that if $w(x,t)$ is the exact
solution to \eqref{eq:fine-scale-wave} when $a_\varepsilon$ is
$\varepsilon$-periodic, then,
\begin{equation} \label{eq:generic_flux_error}
\bigl|
\alpha_{j+\frac{1}{2}}- \bar{a}
\bigr|\leq C_{a,q} \left[(\varepsilon/\eta)^{q}
  +(\varepsilon/\tau)^{q}
\right],
\end{equation}
for a constant $C_{a,q}$ that depends on $a$ and $q$ via ${\mathcal
K}$, but not on $\varepsilon$, $\eta$ and $\tau$. The local averaging
kernel improves the approximation of the homogenized flux by
projecting away variations in the solution and ultimately achieves
high polynomial order $q$ with respect to the size of the domain $\eta$
relative to the period $\varepsilon$. For details on the use
of kernels in HMM we refer to \cite{wave-hmm}. In Appendix
\ref{appendix:kernel}, we provide a new procedure to generate and
evaluate kernels of this type using orthogonal polynomials. This
procedure is efficient and numerically stable even for very high
order kernels and therefore preferable to the procedure provided in
\cite{wave-hmm}.

We can now simplify the expressions for the macro-scale equation.
Given the definition of $P_{j+\frac{1}{2}}(t)$ in \eqref{eq:Pjdef}
this implies that
\begin{equation*}
\frac{F_{j+\frac{1}{2}}(t) - F_{j-\frac{1}{2}}(t)}{H}
= \frac{\alpha_{j+\frac{1}{2}}\bar{w}_{j+1}(t)
  -(\alpha_{j+\frac{1}{2}}+\alpha_{j-\frac{1}{2}})\bar{w}_{j}(t)+
\alpha_{j-\frac{1}{2}}\bar{w}_{j-1}(t)}{H^2},
\end{equation*}
If we let
\begin{align*}
\vec{\bar{w}} &= (\bar{w}_1,\ldots,\bar{w}_{M_x})^T, &
\vec{f}(t) &= (f(t,X_1),\ldots, f(t,X_{M_x}))^T, \\
\vec{v}_j &= (v_j(X_1), \ldots, v_j(X_{M_x}))^T, &
\bar{A} &= \diag{\alpha_{1/2},\ldots,\alpha_{M_x+1/2}},
\end{align*}
we can write the macro-scale solver \eqref{eq:wave-semi-discrete} on
matrix form as
\begin{gather}
\frac{d^2 \vec{\bar{w}}}{d t^2} =
D_{-} \bar{A} D_+\vec{\bar{w}} + \vec{f}, \qquad t > 0,
\label{eq:semi-discrete-wave} \\
\vec{\bar{w}}(0) = \vec{v}_0, \qquad \vec{\bar{w}}_t(0) = \vec{v}_1, \notag
\end{gather}
where $D_{-}$ and $D_+$ are the standard backward and forward
difference operators assuming homogeneous Dirichlet boundary conditions on matrix form.
For other discretizations and in
higher dimensions the matrix $D_{-} \bar{A} D_+$ will be replaced by
a more general matrix approximating the homogenized
operator. Moreover, if we select $D_{-} = -D_{+}^T$ in higher
dimensions, e.g., a staggered grid discretization, then $D_{-}
\bar{A} D_+$ will be symmetric because $\bar{A}$ is always symmetric
when $a_\varepsilon$ is a scalar or symmetric tensor. In the
one-dimensional setting with periodic $a_\varepsilon$ we can pick $H$ so that all
$\alpha_{j+1/2}$ are the same, see \cite{wave-hmm}, and the equations
can be further simplified. We summarize this case in the following
proposition from \cite{wave-hmm}.
\begin{prop}\label{prop:hmmprop}
Consider the one-dimensional case and suppose
$a_\varepsilon(x) = a(x/\varepsilon)$
where $a$ is $1$-periodic. If $H$ is an integer multiple of
$\varepsilon$, the semi-discrete HMM method for \eqref{eq:wave} can be written,
\begin{gather*}
  \frac{d^2 \vec{\bar{w}}}{d t^2} =
  \alpha D^2\vec{\bar{w}} + \vec{f}, \qquad t > 0, \\
  \vec{\bar{w}}(0) = \vec{v}_0, \qquad \vec{\bar{w}}_t(0) = \vec{v}_1,
\end{gather*}
where $\alpha$ is a scalar constant, and $D^2$ is the standard
3-point finite difference Laplacian. Moreover, when the micro-problem
is solved exactly and an averaging kernel in
$\mathbb{K}^{p,q}$ is used,
\[
  |\alpha - \bar{a}| \leq C_{a,q} \left[ (\varepsilon/\eta)^q +
  (\varepsilon / \tau)^q \right],
\]
for a constant $C_{a,q}$ that is independent of $\varepsilon$,
$\eta$ and $\tau$,
but depends on $a$ and $q$.
\end{prop}

\subsection{The WaveHoltz Method} \label{sec:waveholtz-continuous}

WaveHoltz is an iterative methods for solving the Helmholtz equation
\begin{gather} \label{eq:helmholtz2}
\nabla \cdot (a(x) \nabla u) + \omega^2 u = b(x), \qquad x\in\Omega.
\end{gather}
Let $v^{(\ell)}$ denote iterate $\ell$ and $\Pi$ the WaveHoltz
iteration operator, so that
\[
{v}^{(\ell+1)} = \Pi {v}^{(\ell)}.
\]
In each iteration, let $w$ be the solution to an associated
time-dependent wave equation with $v^{\ell}$ as initial data and
$-b(x)\cos(\omega t)$ as forcing function,
\begin{gather}
w_{tt} = \nabla\cdot(a(x)\nabla w) - b(x)\cos(\omega t), \label{eq:WH-wave-eq} \\
w(x,0) = v^{(\ell)}(x), \quad w_t(x,0) = 0. \notag
\end{gather}
Then the next iterate is given by filtering $w$ in time over one
period $T = \frac{2\pi}{\omega}$,
\[
{v}^{(\ell+1)} =\Pi {v}^{(\ell)} = \frac{2}{T}\int_0^T
\left(\cos(\omega t) - \frac{1}{4}\right) {w}(x,t)\, dt.
\]
The convergence in $L^2$ and $H^1$ of the sequence ${v}^{(\ell)}$ to ${u}$, the solution
to \eqref{eq:helmholtz2}, has been established in
\cite{WaveHoltz,rotem2024} under the condition that $\omega$ is not resonant, i.e.\mbox{} not
an eigenvalue of $-\nabla \cdot (a(x) \nabla u)$.

\subsection{The WaveHoltz Heterogeneous Multiscale Method}
\label{sec:discrete-method}
The WaveHoltz Heterogeneous Multiscale Method presented here is a
finite difference method built upon the existing finite difference
HMM for the wave equation \cite{wave-hmm}, as described in
Section~\ref{sec:wave-hmm-continuous}. The key idea is to use the
coarse scale wave equation
\eqref{eq:eff-wave-eq} as the associated wave equation \eqref{eq:WH-wave-eq} in
WaveHoltz. The iterations will then converge to an approximation of
the homogenized Helmholtz equation \eqref{eq:hom-helmholtz}.

We consider the method in one dimension. On the macro-scale we use
the same coarse grid as in Section \ref{sec:macro}. Starting with any
initial guess for the solution $\vec{v}^{(0)} \in \R^{M_x}$, the
WaveHoltz Heterogeneous Multiscale Method produces a sequence of
approximate solutions $\vec{v}^{(\ell)}$ via the fixed point iteration:
\begin{equation} \label{eq:WHMM-fpi}
\vec{v}^{(\ell+1)} = \Pi_h \vec{v}^{(\ell)}.
\end{equation}
Here, $\Pi_h$ is the discretization of the operator $\Pi$ described
in Section \ref{sec:waveholtz-continuous}. In each WaveHoltz
iteration, we solve \eqref{eq:semi-discrete-wave}, which discretizes
\eqref{eq:eff-wave-eq}, over one period in time with $\vec{v}_0 =
\vec{v}^{(\ell)}$, $\vec{v}_1 = 0$, and $\vec{f} = -\vec{b}\cos(\omega t)$. Let $\vec{\bar{w}}$ and
$\bar{A}$ be defined as in Section~\ref{sec:micro-solver}. Hence at
iteration $\ell$, we consider
\begin{gather*}
\frac{d^2 \vec{\bar{w}}}{d t^2} = D_{-} \bar{A} D_{+} \vec{\bar{w}} -
\vec{b}\cos(\omega t), \qquad 0 \le t \le T, \\
\vec{\bar{w}}(0) = \vec{v}^{(\ell)}, \qquad \vec{\bar{w}}_t(0) = 0.
\end{gather*}
Here
$\vec{b}=( b(X_1),\ldots,b(X_{M_x}))^T$
where $b$ is the right hand side of \eqref{eq:hom-helmholtz}.

We discretize in time with time-step $K = T / M_t$ (for a total of
$M_t$ time steps) via the modified time-stepping scheme \cite{rotemThesis, ElWaveHoltz}:
\begin{subequations} \label{eq:discrete-wave}
\begin{gather}
  \vec{\bar{w}}^{n+1} - 2 \vec{\bar{w}}^{n} + \vec{\bar{w}}^{n-1} =
  \theta^2 (D_{-} \bar{A} D_{+} \vec{\bar{w}}^{n} - \vec{b}\cos(n \omega
  K)), \qquad \theta = \frac{\sin(\omega K/2)}{\omega / 2}, \\
  \vec{\bar{w}}^{0} = \vec{v}^{(\ell)}, \qquad \vec{\bar{w}}^1 =
  \vec{\bar{w}}^0 + \tfrac{1}{2}\theta^2 (D_{-} \bar{A} D_{+}
  \vec{\bar{w}}^0 - \vec{b}),
\end{gather}
\end{subequations}
where $\vec{\bar{w}}^n\approx \vec{\bar{w}}(t^n)$ and $t^n=nK$. This
time-stepping scheme differs from the standard leap-frog scheme in
the introduction of $\theta \approx K$ which ensures that
$\vec{\bar{w}}^n = \vec{v}_h \cos (\omega t^n)$ when $\vec{\bar{w}}^0
= \vec{v}_h$ and
$\vec{v}_h$ is
the solution of the discrete Helmholtz equation,
\begin{equation}\label{eq:discHH}
D_{-} \bar{A} D_{+} \vec{v}_h + \omega^2 \vec{v}_h = \vec{b}.
\end{equation}
Finally, the discrete WaveHoltz operator $\Pi_h$ is computed as a left Riemann
sum in time of the solution $\vec{\bar{w}}^n$:
\begin{equation}\label{eq:discrete-pi}
\Pi_h \vec{v}^{(\ell)} = \frac{2}{T}\sum_{n=0}^{M_t-1} K\left(
\cos(\omega t^n) - \rho_0 \right) \vec{\bar{w}}^n
= \frac{2}{M_t} \sum_{n=0}^{M_t-1} \left( \cos\left(\frac{2 \pi
n}{M_t}\right) - \rho_0 \right) \vec{\bar{w}}^n.
\end{equation}
Note that in the continuous setting, $\rho_0 = \frac{1}{4}$. In the
discrete setting $\rho_0 = \frac{1}{4}-\frac{1}{4}\tan^2 \frac{\omega
K}{2}$ is modified to guarantee that the iteration is contractive.
This implies, that $\vec{v}^{(\ell)}$ converges to the unique fixed
point of the discrete WaveHoltz iteration $\vec{v}_h = \Pi_h
\vec{v}_h$. We summarize the main points in the proposition below,
which follows from the results in \cite{WaveHoltz,rotem2024, rotemThesis}.

\begin{prop}\label{prop:whiprop}
Let $-\lambda_j^2$ be the eigenvalues of the matrix $D_{-} \bar{A}
D_{+}$ and suppose
$$
\min_j |\lambda_j-\omega| >0.
$$
Then there exists a unique fixed point $\vec{v}_h$ to $\Pi_h$ satisfying
\begin{equation}\label{eq:whifp}
  \vec{v}_h = \Pi_h \vec{v}_h
  \qquad\text{and}\qquad
  D_{-} \bar{A} D_{+} \vec{v}_h+\omega^2\vec{v}_h=\vec{b}.
\end{equation}
Moreover, for stable timesteps $K$ the WaveHoltz iterates
$\vec{v}^{(\ell)}$ converge to $\vec{v}_h$.
\end{prop}
\begin{remark}
Proposition \ref{prop:whiprop} is stated in terms of the specific one
dimensional discretization matrix $D_{-} \bar{A} D_{+}$, but the
result holds for all stable discretizations and in any dimension.
\end{remark}

\subsection{Approximating the Homogenized Coefficient
Offline}\label{sec:offline-coefficient}
To reduce the computational cost of the method we can approximate the
homogenized coefficient in an offline stage before starting the
WaveHoltz iteration. This coefficient can be reused for several
frequencies $\omega$ and several right hand sides $b$. Recall that HMM estimates
\[ F_{j + \frac{1}{2}}^n = \alpha_{j+\frac{1}{2}} \tfrac{\bar{\vec{w}}^n_{j+1} -
\bar{\vec{w}}^n_{j}}{H} \approx \bar{a} \bar{w}_x({X_{j+\frac{1}{2}}}). \]
By substituting the particular function $\bar{w}(x) = x$, we get that
\[ F_{j + \frac{1}{2}}^n = \alpha_{j+\frac{1}{2}} \approx
\bar{a}(x_{j+\frac{1}{2}}). \]
Moreover, since the homogenized coefficient does not depend on time,
we can compute the homogenized coefficient offline which eliminates
the need to apply HMM at every time-step of the WaveHoltz iteration.
From homogenization theory, if $a_\varepsilon(x) = a(x/\varepsilon)$
is periodic, then $\bar{a}$ is constant. Hence, in this case, we need
only compute a single value of $\bar{a}$.

\section{Analysis in One Dimension} \label{sec:1d-analysis}
In this section, we analyze WHMM applied to the one dimensional
problem with periodic coefficients. For this problem, we can describe
the effective equation estimated by WHMM based on the analysis in
\cite{wave-hmm} and provide precise error estimates.

Consider the Helmholtz equation in one dimension:
\begin{subequations}
  \label{eq:helmholtz-1d}
\begin{gather}
  \left(a_\varepsilon(x) u_x \right)_x + \omega^2 u = b(x), \qquad
  x\in (0,L), \\
  u(0) = u(L) = 0.
\end{gather}
\end{subequations}
Assume that $a_{\varepsilon}(x) = a(x/\varepsilon)$ is
$\varepsilon$-periodic. Then the homogenized coefficient $\bar{a}$ is
constant and the homogenized solution $\bar{u}$ satisfies
\begin{subequations}
  \label{eq:hom-1D-helmholtz}
\begin{gather}
  \bar{a}\bar{u}_{xx} + \omega^2 \bar{u} = b(x), \qquad x\in (0,L), \\
  \bar{u}(0) = \bar{u}(L) = 0.
\end{gather}
\end{subequations}
Applying the same finite difference discretization as in Proposition
\ref{prop:hmmprop}  to \eqref{eq:hom-1D-helmholtz}, we get
\begin{equation} \label{eq:discrete-helmholtz-1d}
\bar{a} D_-D_+ \vec{\bar{u}}_h + \omega^2 \vec{\bar{u}}_h = \vec{b}.
\end{equation}
We first aim to compare $\vec{\bar{u}}_h$, the finite difference
approximation, to $\vec{v}_h$, the fixed point of the WHMM iteration
\eqref{eq:WHMM-fpi}. Following the conclusions in Proposition
\ref{prop:hmmprop} and \ref{prop:whiprop}, we find an effective
equation for the discrete error $\vec{v}_h - \vec{\bar{u}}_h$, and,
by stability estimates provided in Lemma \ref{lem:stability}, we
arrive at the error estimates in Theorem \ref{thm:discrete-error}. In
Theorem \ref{thm:main_theorem} we then also derive an estimate for
the error $\vec{v}_h-\bar{u}$ comparing the WHMM  solution to the
solution to the continuous homogenized problem \eqref{eq:hom-1D-helmholtz}.

Although a unique solution
$\bar{u}$ to \eqref{eq:hom-1D-helmholtz} is
guaranteed to exist in
$H^2(0,L)\cap H^1_0(0,L)$
when $b\in L^2(0,L)$ and $\omega$ is away from resonances, we will
for simplicity use the stronger assumption that
$b \in C^2_b(0,L)$, the space of two times continuously differentiable
functions with bounded derivatives. This assumption implies that
$\bar{u}\in C_b^4(0,L)$. To obtain relative error estimates we will
require additionally that
$b \in C^2_b(0,L)\cap C_0([0,L])$, where $C_0([0,L])$
are the continuous functions that vanish at $x=0$ and $x=L$.

The error estimates are provided in terms of a discrete $L^2$-norm
and a continuous $H^2$-norm which we now define. For a discretization
with $M_x+2$ points as described in Section
\ref{sec:discrete-method}, we define the discrete $L^2$-norm of a
grid function $\vec{f} : f_j = f(X_j)$ as
\[ \|\vec{f}\|_H = \sqrt{\sum_{j=1}^{M_x} H f_j^2}
= H^{\frac{1}{2}}\|\vec{f}\|_2 \approx \|f\|_{L^2}. \]
Throughout, $\|\cdot\|_{L^2} = \|\cdot\|_{L^2(0,L)}$ unless the
domain is explicitly stated. Similar notation is used for the other
Sobolev norms.

In Theorem \ref{thm:main_theorem}, we present the error in terms of an
$H^2$-norm where the second the derivative is scaled with the physical frequency
$\kappa = \omega/\sqrt{\bar{a}}$,
\begin{equation*}
\|f\|_{H^2_\kappa}^2 = \|f\|_{L^2}^2+ \kappa^{-4} \|f_{xx}\|_{L^2}^2
\lesssim \|f\|_{H^2}^2.
\end{equation*}
For fixed $\omega > 0$ and $\bar{a}$, this norm is equivalent to the
standard $H^2$-norm.

Standard interpolation estimates (see Proposition 1.5 in
\cite{Ern2004}) give that
\[
{\|I_H\vec{f}-f\|_{L^2}} \leq (H/\pi)^2\|f_{xx}\|_{L^2},
\]
when $I_H\vec{f}$ is the continuous piecewise linear interpolant of $\vec{f}$.
Together with the bound $\|\vec{f}\|_{H}\leq
\sqrt{3}\|I_H\vec{f}\|_{L^2}$ this leads to the following two
inequalities that we will use in the proofs below,
\begin{subequations}\label{eq:interpineq}
\begin{align}
\|\vec{f}\|_{H}
    &\leq \sqrt{3}\|f\|_{L^2}+\frac{\sqrt{3}H^2}{\pi^2}\|f_{xx}\|_{L^2}, \label{eq:interp-first} \\
\|f\|_{L^2}
    &\leq \|\vec{f}\|_{H}+\frac{H^2}{\pi^2}\|f_{xx}\|_{L^2}. \label{eq:interp-second}
\end{align}
\end{subequations}
Inequality \eqref{eq:interp-first} holds for \(f\in C_b^2(0,L)\), and \eqref{eq:interp-second} holds for
\(f\in C_b^2(0,L)\cap C_0([0,L])\).

Theorems \ref{thm:discrete-error} and \ref{thm:main_theorem} rely on
a few assumptions. As before, let $\{-\lambda_j^2\}_{j=1}^{M_x}$ be
the eigenvalues of $D^2$. Moreover, define the eigenvalue gaps
\begin{gather*}
\delta_a := \min_{j=1,\ldots,M_x} \left| \omega - \sqrt{\bar{a}}
\lambda_j\right|,
\quad
\delta_\alpha := \min_{j=1,\ldots,M_x} \left| \omega -
\sqrt{\alpha} \lambda_j\right|,
\quad
\delta_0 := \min_{k\in \N} \left| \omega - \sqrt{\bar{a}}
\tfrac{\pi k}{L} \right|,
\\
\delta := \min\{\delta_a, \delta_\alpha,\delta_0\}.
\end{gather*}
We then assume
\begin{enumerate}[label=(\textbf{A\arabic*})]
\item\label{assumption:H=n*eps} $H$ is an integer multiple of
  $\varepsilon$ and $\eta = \tau$.

\item\label{assumption:delta} $0<\delta  \leq \omega$, and $\omega \geq 1$.

\item\label{assumption:resolution} Sufficient grid resolution, $H
  \omega \leq \sqrt{\bar{a}}$.
\item\label{assumption:skokram} Sufficient upscaling approximation,
  $\frac{|\bar{a}-\alpha|}{\bar{a}}\leq \frac12.$

\end{enumerate}
Note, due to Proposition~\ref{prop:whiprop} the last assumption
\ref{assumption:skokram}
is essentially a condition on the size of the HMM micro box, more
precisely that $(\varepsilon/\eta)^q$ is sufficiently small.
\begin{theorem}[Discrete Approximation] \label{thm:discrete-error}
Let $\vec{\bar{u}}_h$
be the solution to 
\eqref{eq:discrete-helmholtz-1d}.
Let $\vec{v}_h$ be the fixed-point of WHMM
in \eqref{eq:whifp}
for the same $\vec{b}$
as in \eqref{eq:discrete-helmholtz-1d}. Assume \ref{assumption:H=n*eps},
\ref{assumption:delta}
and \ref{assumption:skokram}. Then the following absolute error estimates hold,
\begin{equation} \label{eq:discrete-absolute-bound}
  \|\vec{\bar{u}}_h - \vec{v}_h\|_H \leq  \frac{C_1}{\delta^2}
  \left(\frac{\varepsilon}{\eta}\right)^q \|\vec{b}\|_H,\qquad
  \|\vec{\bar{u}}_h - \vec{v}_h\|_H \leq \frac{C_2}{\omega^2
  \delta^2} \left(\frac{\varepsilon}{\eta}\right)^q \|D^2 \vec{b}\|_H.
\end{equation}
Similarly, this relative error estimate holds
\begin{equation} \label{eq:discrete-relative-bound}
  \frac{\|\vec{\bar{u}}_h - \vec{v}_h\|_H}
  {\|\vec{\bar{u}}_h\|_H}
  \leq C_3 \frac{\omega}{\delta} \Bigl(\frac{\varepsilon}{\eta}\Bigr)^q.
\end{equation}
The constants $C_1$, $C_2$ and $C_3$ only depend on $a$ and $q$.
\end{theorem}

\begin{remark}
Since Theorem~\ref{thm:discrete-error} is a purely discrete
comparison, the only source of error is the approximation of the
homogenized coefficient $\bar a$ by
$\alpha$. This coefficient error can be made arbitrarily small
by choosing a sufficiently large relative micro-box size
$\eta/\varepsilon$. The eigenvalue gap $\delta$ measures the distance
to resonance and the error bounds deteriorate as $\delta\to
0$. In the absolute estimates
\eqref{eq:discrete-absolute-bound}, the
coefficient approximation error is amplified by the factor
$\delta^{-2}$. In the relative estimate 
\eqref{eq:discrete-relative-bound}, on the other hand, the
coefficient approximation error is amplified by the factor
$\omega/\delta$.
Since the proof of Theorem~\ref{thm:discrete-error} is largely
algebraic, extensions to higher dimensions should be possible,
although some care may be required in handling mixed derivative terms
such as $\partial^2/(\partial x\,\partial y)$, which necessarily
appear when $\bar a$ is anisotropic.
\end{remark}

\begin{proof}[Proof of Theorem \ref{thm:discrete-error}]
Set $A=\bar{a} D^2 + \omega^2 I$ and $B=\alpha D^2 + \omega^2 I$. Due
to \ref{assumption:delta} these matrices are both invertible since
$\delta_a,\delta_\alpha\geq \delta>0$ and by
\eqref{eq:discrete-helmholtz-1d} and \eqref{eq:whifp} in
Proposition~\ref{prop:whiprop}, we have $\vec{\bar{u}}_h
=A^{-1}\vec{b}$ and $\vec{{v}}_h =B^{-1}\vec{b}$. Moreover, $A$ and
$B$ commute, so that
\begin{align}
  \vec{\bar{u}}_h - \vec{v}_h
  &= \left[A^{-1} - B^{-1} \right]\vec{b}
  = A^{-1} B^{-1}(B - A) \vec{b} \notag \\
  &= (\alpha - \bar{a})A^{-1} B^{-1} D^2 \vec{b} \label{eq:err1} \\
  &= (\alpha - \bar{a})A^{-1} B^{-1}\frac{B-\omega^2I}{\alpha} \vec{b}
  = \frac{\alpha - \bar{a}}{\alpha}(I-\omega^2B^{-1}) A^{-1}\vec{b} \notag \\
  &= \frac{\alpha - \bar{a}}{\alpha}(I-\omega^2B^{-1})
  \vec{\bar{u}}_h. \label{eq:err2}
\end{align}
The error estimates now follow from \eqref{eq:err1} and
\eqref{eq:err2} by using Lemma~\ref{lem:stability}, which says that
both $\|A^{-1}\|_H$ and $\|B^{-1}\|_H$ are bounded by
$1/\omega\delta$, as well as Proposition~\ref{prop:hmmprop}, which
says that $|\bar{a}-\alpha|\leq C_{a,q}(\varepsilon/\eta)^q$ when
\ref{assumption:H=n*eps} holds. Thus, from \eqref{eq:err1},
\begin{equation*}
  \begin{aligned}
    ||\vec{\bar{u}}_h - \vec{v}_h||_H
    &\leq
    \frac{|\bar{a}-\alpha|}{\alpha}\|A^{-1}\|_H\|B^{-1}\|_H\|D^2\vec{b}\|_H \\
    &\leq \frac{2C_{a,q}}{\bar{a}}
    \frac{(\varepsilon/\eta)^q}{\omega^2\delta^2}\|D^2\vec{b}\|_H.
  \end{aligned}
\end{equation*}
where we also used the fact that
$\alpha\geq \bar{a}/2$, which follows from
\ref{assumption:skokram}. This shows the second estimate in
\eqref{eq:discrete-absolute-bound} with $C_2=2C_{a,q}/\bar{a}$. Next,
from \eqref{eq:err2}, and the fact that $\omega\geq\delta$ by
\ref{assumption:delta},
\begin{multline}
  ||\vec{\bar{u}}_h - \vec{v}_h||_H\leq
  \frac{|\alpha - \bar{a}|}{\alpha}(1+\omega^2\|B^{-1}\|_H)
  \|\vec{\bar{u}}_h\|_H \\
  \leq
  \frac{C_{a,q}}\alpha (\varepsilon/\eta)^q\left(1+\frac{\omega}{\delta}\right)
  \|\vec{\bar{u}}_h\|_H
  \leq
  \frac{4C_{a,q}}{\bar{a}} (\varepsilon/\eta)^q
  \frac{\omega}{\delta}\|\vec{\bar{u}}_h\|_H,
\end{multline}
which shows \eqref{eq:discrete-relative-bound} with
$C_3=4C_{a,q}/\bar{a}$. Finally, the first estimate in
\eqref{eq:discrete-absolute-bound} follows with $C_1=C_3$ since
$\|\vec{\bar{u}}_h\|_H =\|A^{-1}\vec{b}\|_H\leq
\|\vec{b}\|_H/\omega\delta$. This concludes the proof.
\end{proof}
We now pass from this purely discrete comparison to an estimate against the exact homogenized solution, which adds the finite difference discretization error to the coefficient approximation error.

\begin{theorem}[Continuous Approximation]\label{thm:main_theorem}
Let $\bar u$ be the solution of \eqref{eq:hom-1D-helmholtz} for some
$b\in C_b^2([0,L])$, and denote by
$\bar{\vec u}=(\bar u(X_1),\ldots,\bar u(X_{M_x}))^T$.
Let $\vec{v}_h$ be the fixed-point of WHMM
in \eqref{eq:whifp} for
the vector 
$\vec{b}=( b(X_1),\ldots,b(X_{M_x}))^T$.
Assume
\ref{assumption:H=n*eps}, \ref{assumption:delta},
\ref{assumption:resolution} and
\ref{assumption:skokram}.
 Then the following absolute error estimate holds
\begin{equation}\label{eq:cont-abs-estimate}
  \|\vec{\bar{u}} - \vec{v}_h\|_H \leq C_{\rm abs} \left[
  \frac{(\varepsilon/\eta)^q + H^2\omega^2}{\delta^2} \right]
  \|b\|_{H^2_\kappa},
\end{equation}
where $C_{\rm abs}$ only depends on $a$ and $q$. Furthermore,
suppose also that
\begin{equation}\label{eq:omegaHassump}
  \min\Bigl(H^2\omega^3, \omega (\varepsilon/\eta)^q\Bigr)\leq \delta.
\end{equation}
If $b\in C_b^2(0,L)\cap C_0([0,L])$, there exists another constant
$C_{\rm rel}>0$ that depends on $a$ and $q$, but is independent of
$\omega$ and $b$, such that the following relative error estimate holds
\begin{equation}\label{eq:cont-rel-estimate1}
  \frac{\|\vec{\bar{u}}- \vec{v}_h\|_H}{\|\vec{\bar{u}}\|_H}
  \leq C_{\rm rel}\frac{\omega(\varepsilon/\eta)^q +H^2 \omega^3
  }{\delta}, \qquad
  \forall \omega\geq \omega_0,
\end{equation}
where $\omega_0$ depends on $b$. Finally, assume
$b(x)=g((2x/L-1)\omega)$ where $g\in C^3({\mathbb R})$ with
supp\,$g\subset(-1,1)$, and that
\begin{equation}\label{eq:omegaHassump2}
  H\omega \le \sqrt{\bar{a}}\sqrt{\frac{5}{\mathcal D_g+1}},
\end{equation}
where
${\mathcal D}_g$ is
the $g$-dependent constant in Lemma~\ref{lem:invstability}. Then
\eqref{eq:cont-rel-estimate1} holds with $\omega_0=1$. In this case
$C_{\rm rel}$ remains independent of $\omega$, but may now depend on
$g$, and as before on $a$ and $q$.

\end{theorem}

\begin{remark}
In contrast to the discrete discrete comparison
in Theorem~\ref{thm:discrete-error}, Theorem~\ref{thm:main_theorem} compares the WHMM
solution with the exact homogenized solution. The error therefore has
two sources: the coefficient approximation error, of order
$(\varepsilon/\eta)^q$, and the finite difference discretization
error, of order $H^2\omega^2$. Similar to
Theorem~\ref{thm:discrete-error}, both contributions to the absolute
error in \eqref{eq:cont-abs-estimate} are amplified by the factor
$\delta^{-2}$, where $\delta$ is the eigenvalue gap. The relative
estimate in \eqref{eq:cont-rel-estimate1} shows in addition that the
same two error contributions are amplified by the factor
$\omega/\delta$, in the same way as in
Theorem~\ref{thm:discrete-error}. In particular, the finite
difference contribution appears as $H^2\omega^3$, which has the
standard pollution-error form for
a second order Helmholtz discretization. In the
general case, the relative estimate only holds for $\omega\ge
\omega_0$, where $\omega_0$ depends on $b$. Essentially, a sufficient
portion of the energy of $b$ must be contained in frequencies below
$\omega_0$; compare with Remark~\ref{remark:counter-ex}. For alternative
assumptions that lead to relative error estimates, we refer to
\cite{gander2025fourieranalysisfinitedifference}. The structured case
$b(x)=g((2x/L-1)\omega)$ is useful, since it yields a relative estimate valid
for all $\omega\ge 1$. This class includes, for example, localized
smooth right-hand sides such as mollified delta-type source terms
obtained by scaling a compactly supported $g$. In this case,
however, the constant in the relative estimate may also depend on the profile $g$.
\end{remark}

\begin{proof}[Proof of Theorem \ref{thm:main_theorem}]

We write the error $\vec{\bar{u}} - \vec{v}_h = (\vec{\bar{u}} -
\vec{\bar{u}}_h) + (\vec{\bar{u}}_h - \vec{v}_h)$,
where $\vec{\bar{u}}_h$ is the solution
to \eqref{eq:discrete-helmholtz-1d}
with the same $\vec{b}$ as in the
theorem statement.
By the triangle inequality
\begin{equation} \label{eq:triangle-ineq}
  \| \vec{\bar{u}} - \vec{v}_h\|_H \leq \|\vec{\bar{u}} -
  \vec{\bar{u}}_h\|_H + \|\vec{\bar{u}}_h - \vec{v}_h\|_H =: E_1+E_2.
\end{equation}
For $E_2$ we get from
\eqref{eq:discrete-relative-bound} in Theorem~\ref{thm:discrete-error},
that
\begin{equation}\label{eq:E2a}
  E_2 \leq C_3 \frac{\omega}{\delta} \Bigl(\frac{\varepsilon}{\eta}\Bigr)^q
  \|\vec{\bar{u}}_h\|_H.
\end{equation}
Moreover, by \eqref{eq:interp-first} and the fact that
$H\omega\leq \sqrt{\bar{a}}$,
\begin{multline}\label{eq:H-norm-error}
  \|\vec{b}\|_H \leq
  \sqrt{3}\|b\|_{L^2} + \frac{\sqrt{3}}{\pi^2}H^2 \|b_{xx}\|_{L^2} \\
  =    \sqrt{3}
  \left(\|b\|_{L^2} + \frac{H^2\omega^2}{\bar{a}\pi^2}
  \frac{\bar{a}}{\omega^2} \|b_{xx}\|_{L^2}\right)
  \leq \sqrt{6}\|b\|_{H^2_\kappa}.
\end{multline}
Therefore, using \eqref{eq:discrete-absolute-bound} in
Theorem~\ref{thm:discrete-error},
\begin{equation}\label{eq:E2b}
  E_2 \leq \frac{C_1 }{\delta^2} \left(\frac{\varepsilon}{\eta}\right)^q
  \|\vec{b}\|_H
  \leq
  \frac{C_1\sqrt{6} }{\delta^2} \left(\frac{\varepsilon}{\eta}\right)^q
  \|b\|_{H^2_\kappa}.
\end{equation}
Next, we consider $E_1$. Define the right hand side $\tilde{\vec{b}}$
when the finite difference discretization in
\eqref{eq:discrete-helmholtz-1d} is applied to the exact solution values
$\vec{\bar{u}}$,
\begin{equation} \label{eq:fd-helmholtz-ubar}
  (\bar{a} D^2 + \omega^2 I) \vec{\bar{u}} = \tilde{\vec{b}}.
\end{equation}
Subtracting \eqref{eq:discrete-helmholtz-1d} from
\eqref{eq:fd-helmholtz-ubar}, we get:
\begin{equation}
  (\bar{a}D^2 + \omega^2 I) \left( \vec{\bar{u}} - \vec{\bar{u}}_h
  \right) = \tilde{\vec{b}}-\vec{b},
\end{equation}
and by the stability estimate in Lemma \ref{lem:stability},
\begin{align}
  E_1\leq \tfrac{1}{\delta\omega} \|\tilde{\vec{b}}-\vec{b}\|_H.
  \label{eq:error-as-r}
\end{align}
We next estimate the elements of the difference
$\tilde{\vec{b}}-\vec{b}$. Since $b\in C_b^2(0,L)$ the PDE
\eqref{eq:hom-1D-helmholtz} gives that
$\bar{u}\in C_b^4(0,L)$. Therefore, by Taylor's theorem,
\[
  \bar{u}(x) = \bar{u}(X_j) + \bar{u}_x(X_j)(x-X_j) +
  \tfrac{1}{2}\bar{u}_{xx}(X_j)(x-X_j)^2 +
  \tfrac{1}{3!}\bar{u}_{xxx}(X_j)(x-X_j)^3 + \Psi^j(x),
\]
where
\[
  \Psi^j(x) = \frac{1}{3!}\int_{X_j}^x (x - s)^3 \bar{u}_{xxxx}(s) \, ds.
\]
From now on, denote by
$\bar{u}_j = \bar{u}(X_j)$,
$(\bar{u}_{xx})_j = \bar{u}_{xx}(X_j)$ and so on.
Then
\begin{align*}
  \tilde{b}_j-b_j&= \bar{a}\frac{\bar{u}_{j+1} - 2\bar{u}_{j} +
  \bar{u}_{j-1}}{H^2} + \omega^2 \bar{u}_j-b_j = \bar{a}\left[
  (\bar{u}_{xx})_j + \frac{\Psi^j_{j+1} + \Psi^j_{j-1}}{H^2} \right]
  + \omega^2 \bar{u}_j -b_j\\
  &= \left( \bar{a}(\bar{u}_{xx})_j + \omega^2 \bar{u}_j -b_j\right)
  + \bar{a} \frac{\Psi^j_{j+1} + \Psi^j_{j-1}}{H^2}
  = \bar{a} \frac{\Psi^j_{j+1} + \Psi^j_{j-1}}{H^2}.
\end{align*}
To estimate $\tilde{b}_j-b_j$, we apply the Cauchy-Schwarz inequality,
\begin{align*}
  3!|\Psi^j_{j-1}+\Psi^j_{j+1}| &=
  \!\! \left|\int_{X_{j-1}}^{X_{j}} (s-X_{j-1})^3 \bar{u}_{xxxx}(s) \, ds
  +\int_{X_{j}}^{X_{j+1}} (X_{j+1} - s)^3 \bar{u}_{xxxx}(s) \, ds
  \right|
  \\
  &\leq \!\!\left[ \int_{X_{j-1}}^{X_{j+1}} \min\left(
  |s-X_{j-1}|,|X_{j+1} - s|\right)^6 \, ds \right]^\frac{1}{2}
  \!\!\left[ \int_{X_{j-1}}^{X_{j+1}} |\bar{u}_{xxxx}(s)|^2
  \, ds \right]^\frac{1}{2} \\
  &=
  \sqrt{\frac27}H^{\frac72}
  \|\bar{u}_{xxxx}\|_{L^2([X_{j-1}, X_{j+1}])}.
\end{align*}
Therefore,
\begin{align}
  \|\tilde{\vec{b}}-\vec{b}\|_H^2
  &= \sum_{j=1}^{M_x} H |\tilde{b}_j-b_j|^2
  = \frac{\bar{a}^2}{H^3}\sum_{j=1}^{M_x}
  |\Psi^j_{j-1}+\Psi^j_{j+1}|^2 \nonumber\\
  &=
  \frac{\bar{a}^2H^4}{126}\sum_{j=1}^{M_x}\|\bar{u}_{xxxx}\|_{L^2([X_{j-1},
  X_{j+1}])}^2
  \nonumber\\
  &=\frac{\bar{a}^2H^4}{126}\sum_{j=1}^{M_x}
  \int_{X_{j-1}}^{X_{j+1}}|\bar{u}_{xxxx}(s)|^2\, ds
  \leq \frac{\bar{a}^2H^4}{63}\|\bar{u}_{xxxx}\|_{L^2}^2.
  \label{eq:bdiff}
\end{align}
Next, we bound $\|\bar{u}_{xxxx}\|_{L^2}$. Differentiating
\eqref{eq:hom-1D-helmholtz} twice, we get:
\[
  \bar{u}_{xxxx} = \frac{b_{xx}-\omega^2 \bar{u}_{xx}}{\bar{a}}
  = \frac{\bar{a} b_{xx} - \omega^2(b - \omega^2 \bar{u})}{\bar{a}^2},
\]
so,
\[
  \bar{a}\|\bar{u}_{xxxx}\|_{L^2}
  \leq \|b_{xx}\|_{L^2} + \tfrac{\omega^2}{\bar{a}}\|b\|_{L^2} +
  \tfrac{\omega^4}{\bar{a}} \|\bar{u}\|_{L^2}
  \leq \tfrac{\omega^2}{\bar{a}}\left(\sqrt{2}\|b\|_{H^2_\kappa} +
  \omega^2 \|\bar{u}\|_{L^2}\right).
\]
Combining this with \eqref{eq:error-as-r} and \eqref{eq:bdiff} we obtain
\begin{equation}\label{eq:E1a}
  E_1
  \leq D_1\tfrac{H^2\omega}{\delta}
  \left(
  \|b\|_{H^2_\kappa} + \omega^2 \|\bar{u}\|_{L^2}\right),
\end{equation}
where $D_1=\sqrt{2/63}/\bar{a}$. From the stability estimate in Lemma
\ref{lem:stability},
\[ \|\bar{u}\|_{L^2} \leq \tfrac{1}{\omega \delta} \|b\|_{L^2}, \]
which, after also using \ref{assumption:delta}, further gives
\begin{align*}
  E_1
  \leq D_1\tfrac{H^2\omega}{\delta}
  \left(
  \|b\|_{H^2_\kappa} + \tfrac{\omega}{\delta} \|b\|_{L^2}\right)
  \leq 2D_1\tfrac{H^2\omega^2}{\delta^2}
  \|b\|_{H^2_\kappa}.
\end{align*}
Together with \eqref{eq:E2b} this proves \eqref{eq:cont-abs-estimate}
with $C=\max(\sqrt{6}C_1,2D_1)$.

To prove the relative estimate \eqref{eq:cont-rel-estimate1} we use
Lemma~\ref{lem:invstability}, 
and assume that $\omega$ is larger than $\omega_0=\omega_b$ in the
general case, and larger than $\omega_0=1$ 
when $b=g((2x/L-1)\omega)$.
We then apply it to $\|b\|_{H^2_\kappa}$ in \eqref{eq:E1a},
\begin{align*}
  E_1
  &\leq D_1
  \tfrac{H^2\omega}{\delta}
  \left(
    {\mathcal D}\omega^2
    \|\bar{u}\|_{L^2}
    + \omega^2\|\bar{u}\|_{L^2}
  \right)
  =D_1 ({\mathcal D} + 1)
  \tfrac{H^2\omega^3}{\delta}
  \|\bar{u}\|_{L^2}.
\end{align*}
where $\mathcal D=4$ in the general case and $\mathcal D=\mathcal
D_g$ in the structured case.
 Then  we note that
$$
\|\vec{\bar{u}}_h\|_{H}\leq
\|\vec{\bar{u}}\|_{H}+\|\vec{\bar{u}}-\vec{\bar{u}}_h\|_{H}
=\|\vec{\bar{u}}\|_{H}+E_1.
$$
Thus, by \eqref{eq:E2a}
\begin{align*}
  E_1+E_2
  &\leq E_1+C_3 \frac{\omega}{\delta} \Bigl(\frac{\varepsilon}{\eta}\Bigr)^q
  \|\vec{\bar{u}}_h\|_H
  \leq E_1+C_3 \frac{\omega}{\delta} \Bigl(\frac{\varepsilon}{\eta}\Bigr)^q
  (\|\vec{\bar{u}}\|_{H}+E_1)\\
  &\leq
  C_3 \frac{\omega}{\delta}
  \Bigl(\frac{\varepsilon}{\eta}\Bigr)^q\|\vec{\bar{u}}\|_{H} + D_1
  ({\mathcal D} + 1)
  \frac{H^2\omega^3}{\delta}
  \left(
  1+C_3\frac{\omega}{\delta} \Bigl(\frac{\varepsilon}{\eta}\Bigr)^q\right)
  \|\bar{u}\|_{L^2}
\end{align*}
To bound $||\bar{u}||_{L^2}$ by $\|\vec{\bar{u}}\|_{H}$ we use the estimate in \eqref{eq:interp-second}, and the fact that
$\bar{u}_{xx}=(b-\omega^2\bar{u})/\bar{a}$ by
\eqref{eq:hom-1D-helmholtz} and Lemma~\ref{lem:invstability},
\begin{multline}
  ||\bar{u}||_{L^2}\leq \|\vec{\bar{u}}\|_{H}+
  \tfrac{H^2}{\pi^2}\|\bar{u}_{xx}\|_{L^2}
  \leq \|\vec{\bar{u}}\|_{H}+
  \tfrac{H^2}{\bar{a}\pi^2}(\|b\|_{L^2}+\omega^2\|\bar{u}\|_{L^2}) \\
  \leq \|\vec{\bar{u}}\|_{H}+
  \tfrac{H^2\omega^2}{\bar{a}\pi^2}({\mathcal D}+1)\|\bar{u}\|_{L^2}.
\end{multline}
That means that since $H\omega\leq \sqrt{\bar{a}}$ in the general case,
$$
||\bar{u}||_{L^2}\leq(1-5/\pi^2)^{-1} \|\vec{\bar{u}}\|_{H}.
$$
By construction, the same inequality holds when $b=g((2x/L-1)\omega)$
because of the additional assumption
\eqref{eq:omegaHassump2}.
Thus
\begin{align*}
  E_1+E_2
  &\leq \left[
    C_3 \frac{\omega}{\delta} \Bigl(\frac{\varepsilon}{\eta}\Bigr)^q
    + D_1 ({\mathcal D} + 1)
    \frac{H^2\omega^3}{\delta}
    \left(
      1+C_3\frac{\omega}{\delta}
  \Bigl(\frac{\varepsilon}{\eta}\right)^q\Bigr) (1-5/\pi^2)^{-1}\right]
  \|\vec{\bar{u}}\|_{H} \\
  &\leq
  C\frac{\omega
  \left(\varepsilon/\eta\right)^q+H^2\omega^3}{\delta}\|\vec{\bar{u}}\|_{H}.
\end{align*}
Here we can take
$C=\max(C_3,C')+C_3C'$
with
$C'=D_1({\mathcal D}+1)(1-5/\pi^2)^{-1})$,
since by \eqref{eq:omegaHassump}
$$
H^2\omega^3\,\omega\Bigl(\frac{\varepsilon}{\eta}\Bigr)^q
\leq \delta
\max\left(
H^2\omega^3,\ \omega\Bigl(\frac{\varepsilon}{\eta}\Bigr)^q\right).
$$
Dividing by $\|\vec{\bar u}\|_H$ completes the proof of
\eqref{eq:cont-rel-estimate1}.
\end{proof}

\subsection{Supporting Lemmas}

In the analysis below we use sine series expansions of functions.
The main points are summarized here.
A function $f\in L^2(0,L)$ can be expanded in sine series
with coefficients $\hat{f}_k$ such that
\begin{subequations}\label{eq:dft}
\begin{gather}
f(x) = \sqrt{\frac{2}{L}}\sum_{k=1}^\infty \hat{f}_k
\sin\left(\frac{k\pi x}{L}\right),
\quad
\hat{f}_k = \sqrt{\frac{2}{L}}\int_0^L f(x) \sin\left(\frac{k\pi
x}{L}\right) dx,
\\
\|f\|_{L^2}^2 = \sum_{k=1}^\infty |\hat{f}_k|^2.
\end{gather}
\end{subequations}
If additionally $f\in H^2(0,L)\cap H^1_0(0,L)$ then
$-(k\pi/L)^2\hat{f}_k$
are the coefficients of $f_{xx}$.

\begin{lemma}[Stability]\label{lem:stability}
Consider \eqref{eq:hom-1D-helmholtz} and its discrete counterpart
\eqref{eq:discrete-helmholtz-1d}. Assume \ref{assumption:delta} and
that $b\in L^2(0,L)$. Then
\begin{equation}\label{eq:discrete-stability}
\|\vec{\bar{u}}_h\|_H \leq \tfrac{1}{\omega \delta}\|\vec{b}\|_H,
\end{equation}
and
\begin{equation}\label{eq:continuous-stability}
\|\bar{u}\|_{L^2} \leq \tfrac{1}{\omega \delta} \|b\|_{L^2}.
\end{equation}
\end{lemma}
\begin{proof}
We start with the discrete approximation. Since $D^2$ is symmetric
and negative definite, it can be diagonalized by an orthogonal matrix
$U$ as $D^2 = -U\Lambda^2U^T$ where $\Lambda_{ii} = \lambda_i > 0$. Thus,
\[ \vec{\bar{u}}_h = U(\omega^2 I - \bar{a} \Lambda^2)^{-1} U^T \vec{b}. \]
Taking the norm of both sides:
\[ \|\vec{\bar{u}}_h\|_H \leq \|(\omega^2 I - \bar{a}
\Lambda^2)^{-1}\|_H \cdot \|\vec{b}\|_H. \]
Since $(\omega^2 I - \bar{a} \Lambda^2)^{-1}$ is diagonal, its 2-norm
(and therefore $H$-norm) is the maximum element in absolute value. Define,
\[ 
\lambda^* = \arg \min_i | \omega - \sqrt{\bar{a}} \lambda_i |. 
\]
Then
\[ \|(\omega^2 I - \bar{a} \Lambda^2)^{-1}\|_H = \frac{1}{| \omega -
\sqrt{\bar{a}} \lambda^* | \cdot| \omega + \sqrt{\bar{a}} \lambda^*|}
= \frac{1}{\delta_a |\omega + \sqrt{\bar{a}} \lambda^*|} \leq
\frac{1}{\delta \omega}. \]
Thus,
\[ \|\vec{\bar{u}}_h\|_H \leq \frac{\|\vec{b}\|_H}{\omega \delta}. \]
This proves \eqref{eq:discrete-stability}. 

We follow a similar
procedure in the continuous setting.
Since $b\in L^2(0,L)$ it follows that
$\bar{u}\in H^2(0,L)\cap H^1_0(0,L)$.
We can therefore expand both functions in sine series with
coefficients $\hat{b}_k$ and $\hat{u}_k$. Moreover, the coefficients
of $\bar{u}_{xx}$ are $-(\pi k/L)^2\hat{u}_k$. We therefore get the
relation $\hat{u}_k= \frac{\hat{b}_k}{\omega^2 - \bar{a}\cdot(\pi k / L)^2}$.
Similar to the discrete case, we get
\begin{align*}
\|\bar{u}\|_{L^2}^2 &= \sum_{k=1}^\infty |\hat{u}_k|^2  =
\sum_{k=1}^\infty \left| \frac{\hat{b}_k}{\omega^2 - \bar{a}\cdot
(\pi k / L)^2} \right|^{2} \leq \frac{1}{\omega^2
\delta_0^2}\sum_{k=1}^\infty |\hat{b}_k|^2 \leq \frac{1}{\omega^2
\delta^2} \|b\|_{L^2}^2.
\end{align*}
Taking the square root of both sides proves
\eqref{eq:continuous-stability}.
We note finally that $\delta>0$
by \ref{assumption:delta},
making the right hand sides well-defined.
This completes the proof.
\end{proof}

The next lemma considers a cumulative energy distribution function.
For non-zero $b\in L^2(0,L)$ we define it based on the sine series
coefficients,
\begin{equation}\label{eq:Phidef}
\Phi_b(\eta) := \frac{\sum_{1\leq k\leq \eta}|\hat{b}_k|^2}{\|b\|_{L^2}^2}.
\end{equation}
This function is used to characterize the right hand side in the
Helmholtz equation.

\begin{lemma}\label{lem:Phiprop}
Suppose $b\in L^2(0,L)$ is real and non-zero. Then $\Phi_b(\eta)$ defined in \eqref{eq:Phidef} is a
nondecreasing function from $[0,\infty)$ to $[0,1]$. It converges to
one when $\eta\to\infty$ and it is scale invariant,
$\Phi_{sb}=\Phi_{b}$ for all real $s\neq 0$.
Assume additionally that 
$g\in H^{1}(\R)$ is real, non-zero
and compactly supported in
$(-1,1)$.
If $b_\omega(x):=g((2x/L-1)\omega)$ and $\omega\geq 1$, there exists
a nondecreasing function $\bar{\Phi}_g:[0,\infty)\to [0,1]$ such
that $\lim_{\alpha\to\infty}\bar{\Phi}_g(\alpha)=1$ and
$$
\Phi_{b_\omega}(\alpha\omega)\geq
\bar\Phi_g(\alpha), \qquad \forall \omega\geq 1.
$$
\end{lemma}
\begin{proof}
The scale invariance and monotonicity follow directly from the
definition, since the terms in the sum are non-negative. The
convergence to one of $\Phi_b$ follows from Parseval's identity. For
the case when $b_\omega(x)=g((2x/L-1)\omega)$, we first note that
\begin{equation*}
\begin{aligned}
  \|b_\omega\|^2_{L^2}
  &=\int_0^L |b_\omega(x)|^2 dx
  =\int_0^L |g((2x/L-1)\omega)|^2 dx \\
  &=\frac{L}{2\omega}\int_{-\omega}^{\omega} |g(y)|^2 dy
  =\frac{L}{2\omega}\|g\|_{L^2({\mathbb R})}^2,
\end{aligned}
\end{equation*}
assuming $\omega\geq 1$. We then compute the sine series coefficients
of $b_\omega$ as
\begin{align*}
\hat{b}_k &= \sqrt{\frac{2}{L}}\int_0^L b(x) \sin\left(\frac{k\pi
x}{L}\right) dx = \frac{\sqrt{2L}}{2\omega}\int_{-\omega}^{\omega}
g(y)\sin\left(\frac{k\pi y}{2\omega}+\frac{k\pi}{2}\right) dy \\
&= \frac{\sqrt{2L}}{k\pi}\int_{-1}^1 g_x(y)\cos\left(\frac{k\pi
y}{2\omega}+\frac{k\pi}{2}\right) dy.
\end{align*}
Therefore, by the Cauchy–Schwarz inequality,
$$
|\hat{b}_k|\leq
\frac{\sqrt{2L}}{k\pi}\int_{-1}^1 |g_x|
\, dx \leq
\frac{2\sqrt{L}}{k\pi}\|g_x\|_{L^2(-1,1)},
$$
and, for $\omega\geq 1$, $\alpha>1$
\begin{align*}
\Phi_{b_\omega}(\alpha\omega)
&= 1-\frac{\sum_{k>\alpha\omega}|\hat{b}_k|^2}{\|b_\omega\|_{L^2}^2}
\geq
1 - \frac{2\omega
\frac{4L}{\pi^2}\|g_x\|^2_{L^2(-1,1)}}{L\|g\|_{L^2(-1,1)}^2}
\sum_{k>\alpha\omega} \frac{1}{k^2}
\\
&\geq 1 - \frac{8\omega\|g_x\|^2_{L^2(-1,1)}}{\pi^2\|g\|_{L^2(-1,1)}^2}
\int_{\alpha\omega-1}^\infty\frac{dx}{x^2}
\geq 1 - \frac{C\omega}{\alpha\omega-1}
\geq 1 - \frac{C}{\alpha-1}.
\end{align*}
We can therefore take
\[
  \bar{\Phi}_g(\alpha)=
  \begin{cases}
    0, & \alpha\leq C+1, \\
    1-\dfrac{C}{\alpha-1}, & \alpha> C+1,
  \end{cases}
\qquad
\text{where}
\qquad
  C=\frac{8\|g_x\|_{L^2(-1,1)}^2}{\pi^2 \|g\|_{L^2(-1,1)}^2}.
\]
This function is nondecreasing, takes values in $[0,1]$, converges to
one as $\alpha\to\infty$, and satisfies
$\bar\Phi_g(\alpha)\le \Phi_{b_\omega}(\alpha\omega)$
for all $\omega\ge 1$.
\end{proof}

\begin{lemma}[Inverse stability]\label{lem:invstability}
All frequencies in this lemma are assumed to satisfy \ref{assumption:delta}.
Let $b\in H^2(0,L)\cap H^1_0(0,L)$ be real and let $\bar u$ be the
solution to \eqref{eq:hom-1D-helmholtz} with right hand side $b$.
There exists
$\omega_b>0$, depending on $b$, such that,
\begin{equation}\label{eq:continuous-invstability}
\|b\|_{H_{\kappa}^2}
\leq
4\omega^2\|\bar u\|_{L^2},\qquad 
\forall\omega\geq \omega_b.
\end{equation}
If
$b(x)=g((2x/L-1)\omega)$ where $g\in H^3({\mathbb R})$ is real with compact
support in $(-1,1)$, there is a real number ${\mathcal D}_g$, which
depends on $g$ but is independent of $\omega$, such that
\begin{equation}\label{eq:continuous-invstability2}
\|b\|_{H_{\kappa}^2} \leq {\mathcal D}_g \omega^2
\|\bar{u}\|_{L^2},\qquad 
\forall\omega\geq 1.
\end{equation}
\end{lemma}

\begin{proof}
Since $b\in H^2(0,L)\cap H^1_0(0,L)$ and $\bar u$ solves
\eqref{eq:hom-1D-helmholtz}, it follows that
$$
\|b\|_{H^2_\kappa}^2=
\sum_{k=1}^\infty \left(1 + \bar{a}^2\left(\frac{k\pi}{\omega
L}\right)^4\right) \hat{b}_k^2,
\qquad
\hat{u}_k = \frac{\hat{b}_k}{\omega^2-\bar{a}\left(\frac{k\pi}{L}\right)^2}.
$$
Let $\gamma \geq  \gamma_0:=\frac{L\sqrt{2}}{\pi\sqrt{\bar a}}$.
We then have
\begin{align*}
\|\bar{u}\|_{L^2}^2 &\geq
\sum_{1\leq k\leq \gamma\omega} \hat{u}_k^2=
\sum_{1\leq k\leq \gamma\omega}
\frac{\hat{b}_k^2}{\left(\omega^2-{2k^2}/{\gamma_0^2}\right)^2}
\\
&\geq 
\min_{1\leq k\leq \gamma\omega}
\frac{1}{\left(\omega^2-{2k^2}/{\gamma_0^2}\right)^2}
\sum_{1\leq k\leq \gamma\omega} \hat{b}_k^2
\geq \frac{1}{D_1(\gamma)^2\omega^4}
\sum_{1\leq k\leq \gamma\omega} \hat{b}_k^2,
\end{align*}
where $D_1(\gamma) =
2(\gamma/\gamma_0)^2-1\geq 1$. Furthermore,
using the
definition of $\Phi_b$ in \eqref{eq:Phidef}, 
\begin{align*}
\sum_{1\leq k\leq \gamma\omega} \hat{b}_k^2
&\geq
\frac{1}{2}\sum_{1\leq k\leq \gamma\omega} \hat{b}_k^2
\left(1 + \left(\frac{k}{\gamma\omega}\right)^4\right)
\geq
\frac12\min\left(1,\left(\frac{L}{\gamma\pi\sqrt{\bar{a}}}\right)^4\right)
\\
&\qquad\times 
\sum_{1\leq k\leq \gamma\omega} 
\left(1 + \bar{a}^2\left(\frac{k\pi}{\omega L}\right)^4\right)\hat{b}_k^2\\
&=
\frac{D_2(\gamma)^2}{2}\left(
  \Phi_b(\gamma\omega)
  \|b\|_{L^2}^2+
  \Phi_{b_{xx}}(\gamma\omega)\frac{\bar{a}^2}{\omega^4}
\|b_{xx}\|_{L^2}^2\right)
\geq
\frac12 D_2(\gamma)^2F(\gamma\omega)
\,\|b\|_{H^2_\kappa}^2,
\end{align*}
where
{\scriptsize
\[
  D_2(\gamma) =
  \min\left(1,\left(\frac{L}{\gamma\pi\sqrt{\bar{a}}}\right)^2\right)
  = \tfrac{1}{2}\left(\frac{\gamma_0}{\gamma}\right)^2\leq \frac12,
\]
and
\[
  F(\omega) = \min\left(\Phi_b(\omega),\Phi_{b_{xx}}(\omega)\right).
\]
}
In conclusion,
$$
\|b\|_{H^2_\kappa} \leq
\sqrt{2}\frac{D_1(\gamma)\omega^2}{D_2(\gamma)F(\gamma\omega)^{1/2}}
\|\bar{u}\|_{L^2}.
$$
We note that since $\Phi_b(\omega)$ and $\Phi_{b_{xx}}(\omega)$ are
nondecreasing and converge to one when $\omega\to\infty$ by
Lemma~\ref{lem:Phiprop}, the same is true for $F$. We can therefore let
$\omega_b$ be such that $F(\gamma\omega)\geq 1/2$ when $\omega\geq \omega_b$
and $\gamma=\gamma_0$. Then $D_1(\gamma)=1$ and $D_2(\gamma)=1/2$.
This  proves \eqref{eq:continuous-invstability}. Observe that
$\omega_b$ depends on $b$ since $F$ depends on $b$.

In the second case, when $b(x)=g((2x/L-1)\omega)$, we note first that
the assumptions on $g$, together with
$\omega\geq1$,
again ensure that $b\in H^2(0,L)\cap H^1_0(0,L)$. Moreover,
$(L/2\omega)^2 b_{xx}(x)=g_{xx}((2x/L-1)\omega)$.
Lemma~\ref{lem:Phiprop} then guarantees the existence of $\bar\Phi_g$
and $\bar\Phi_{g_{xx}}$ such that
\begin{equation*}
\begin{aligned}
  F(\gamma\omega)
  &=\min\left(\Phi_b(\gamma\omega),\Phi_{b_{xx}}(\gamma\omega)\right) \\
  &=\min\left(\Phi_b(\gamma\omega),\Phi_{(L/2\omega)^2b_{xx}}(\gamma\omega)\right)
  \\
  &\geq \min\left(\bar\Phi_g(\gamma),\bar\Phi_{g_{xx}}(\gamma)\right).
\end{aligned}
\end{equation*}
where we also used the scale invariance of $\Phi_{b_{xx}}$. Moreover,
$\bar\Phi_g$ and $\bar\Phi_{g_{xx}}$ are both nondecreasing and
converge to one, so we can choose $\gamma_g\geq \gamma_0$ such that
\[
\min(\bar{\Phi}_g(\gamma_g),\bar{\Phi}_{g_{xx}}(\gamma_g))\geq 1/2.
\]
Then,
\[
F(\gamma_g\omega)\geq 1/2 \qquad \forall \omega\geq 1.
\]
This gives us
\eqref{eq:continuous-invstability2}
with ${\mathcal D}_g=2D_1(\gamma_g)/D_2(\gamma_g)$, which depends on
$g$ only via $\gamma_g$.
\end{proof}

\begin{remark}\label{remark:counter-ex}
With the choice $b(x)=\sin(k\pi x/L)$, we have
\[
\bar{u}(x) = \frac{\sin(k\pi x/L)}{\omega^2-\bar{a}(k\pi/L)^2},
\qquad
\|b\|_{H_\kappa^2} = {\mathcal D}_0\omega^2\|\bar{u}\|_{L^2},
\]
where
\[
{\mathcal D}_0={\mathcal D}_0(k\pi/\omega L)=|1-\bar{a}(k\pi/\omega
L)^2|\left(1 + \bar{a}^2\left(\frac{k\pi}{\omega L}\right)^4\right)^{1/2}.
\]
Hence, for fixed $\omega$ the value ${\mathcal D}_0$ grows with $k$.
This shows that there can be no inverse stability estimate where both
${\mathcal D}_0$ and $\omega_b$ are independent of $b$. In general,
${\mathcal D}_0$ depends on the fraction of the total energy of $b$
that is contained in the first $\omega_b$ frequencies.
\end{remark}

\section{Numerical Experiments in One Dimension} \label{sec:experiments1D}
In this section we consider the model problem:
\begin{subequations} \label{eq:model-1D}
\begin{gather}
\left(a_\varepsilon(x) u_x\right)_x + \omega^2 u = b(x), \qquad x\in (0,1), \\
u(0) = u(1) = 0.
\end{gather}
\end{subequations}
Here, the forcing function $b$ will be taken to be an approximate
point source centered at $x_0 = 0.3$ given by
\[ b(x) = \tfrac{s}{\sqrt{\pi}}e^{-s^2(x-x_0)^2}, \qquad s = 20. \]
In Sections \ref{sec:P1} and \ref{sec:P2} we consider coefficients
$a_\varepsilon$ for which the homogenized coefficient is known,
starting with a simple periodic coefficient for which Theorem
\ref{thm:main_theorem} directly applies.
Throughout Sections \ref{sec:experiments1D} and \ref{sec:experiments2D}, the WaveHoltz fixed-point iteration is accelerated by the conjugate gradient method \cite{Saad-IterativeMethods} to a tolerance of $10^{-12}$ so that effectively no error from the iteration is introduced and the computed solution $\vec{v}_h$ is truly the fixed point of WaveHoltz.
We will compare the results
of WHMM with a highly resolved ($M_x = 2000$) finite difference
solution using the exact homogenized coefficient, and refer to this
solution as the ``exact'' solution $\bar{\vec{u}}$.
In addition, in
all three examples we will compare against a direct numerical
simulation (DNS) of \eqref{eq:model-1D} computed at high resolution.
We denote by $\vec{u}_\varepsilon$ the DNS solution, and take
\[ M_{x, \rm DNS} = \max\{ 2000, \lceil 20/\varepsilon \rceil \}. \]
The relative estimates in Theorem \ref{thm:main_theorem} state that
\[
\frac{\|\vec{v}_h -
\bar{\vec{u}}\|_{H}}{\|\bar{\vec{u}}\|_{H}} = O(H^2 \omega^3/\delta) + O(\omega(\varepsilon/\eta)^2/\delta).
\]
The homogenization ansatz also suggests
\(
\|\bar{u} - u_{\varepsilon}\|_{L^2}/\|u_\varepsilon\|_{L^2} = O(\varepsilon),
\)
so,
\[ \frac{\|I_H \vec{v}_h -
\vec{u}_\varepsilon\|_{L^2}}{\|\vec{u}_\varepsilon\|_{L^2}} = O(\varepsilon)
+ O(H^2 \omega^3/\delta)
+O(\omega(\varepsilon/\eta)^2/\delta). \]
As before $I_H : \R^{M_x} \to H^1(0,L)$ is the piecewise linear interpolation operator on the coarse grid. The $L^2$-norm is approximated by quadrature on the fine DNS grid. 

Each one dimensional numerical experiment is organized as follows: first, we show that the approximation error $|\bar{a} - \alpha|$ controlling the $O((\varepsilon/\eta)^q)$ term can be made small,
and that the $O(H^2)$ approximation error dominates for fixed $\omega$ and $\delta$. We see also that this is not the case if a simple average is used in place of the high-order local averaging kernels. Finally, we allow both $\omega$ and $\varepsilon$ to vary but fix the quantity $H^2\omega^3$ which demonstrates how the approximation error scales with $1/\delta$ and $\varepsilon$.
When fixing the quantity $H^2\omega^3$, we select $H=1/M_x$ according to
\begin{equation} \label{eq:ppw}
M_x = \max\{ 64, C \omega^{\frac{3}{2}} \}.
\end{equation}
Here, we take $C = 4$ which approximately resolves the solution with at
least ten points per wavelength for all values of $\omega$
considered ensuring the approximation error is small clearly showing the $O(\varepsilon)$ homogenization error.

\subsection{Local Averaging Kernels} \label{sec:num-kernel}
Recall from Section \ref{sec:micro-solver} that the flux
$F_{j+\frac{1}{2}}$ is computed as an average of the flux of the
micro problem solution weighted by the local averaging kernel
$\mathcal{K}$ belonging to a family of kernels $\mathbb{K}^{p,q}$. A
kernel $\mathcal{K}:[-1,1]\to\R$ belonging to $\mathbb{K}^{p,q}$ has
$q$ compactly supported derivatives, and it exactly reproduces the
mean of any polynomial of degree at most $p$; that is, the integral
of such a polynomial weighted by $\mathcal{K}$ over $[-1,1]$
evaluates to the average value of the polynomial on that interval.
From these two properties, the order parameters $p,q$ determine the
quality of the approximation. Namely, if
$a_\varepsilon(x)\frac{\partial w}{\partial x}$ is periodic then we
have  that
\[
\bar{a}\frac{\partial\bar{w}}{\partial x} = \frac{1}{\eta\tau}
\int_{-\tau}^{\tau}\int_{-\eta}^{\eta}
\mathcal{K}(x/\eta)\mathcal{K}(t/\tau)
a_\varepsilon(X_{j+\frac{1}{2}}+x) \frac{\partial w}{\partial x} \, dx\, dt
+ O\left( (\varepsilon/\eta)^q \right),
\]
for any $p,q \geq 0$. However, when the integral is approximated by
quadrature, a standard result in approximation states that the
integral is approximated to order $q$ (independently of $p$) because
the integrand has $q$ continuous compactly supported derivatives.
Thus for purely periodic problems it is sufficient to choose $p = 1$
and $q$ to be high order to yield an accurate approximation. When
$a_\varepsilon$ has $p$ continuous derivatives in the slow variable
$x$ and periodic in $x/\varepsilon$, then
$\bar{a}\frac{\partial\bar{w}}{\partial x}$ has $p$ continuous
derivatives. The kernel $\mathcal{K}$ exactly averages all
polynomials of degree $p$ or less, hence the approximation error
gains a term of order $p$ in addition to the term of order $q$. For
these problems we typically choose both $p$ and $q$ to be high order.

In the numerical experiments $K^{(p,q)}$ is the unique polynomial of
minimum degree polynomial in $\mathbb{K}^{p,q}$ computed using the
procedure in Appendix \ref{appendix:kernel}.

\input{experiments/Periodic1D/periodic1d}
\input{experiments/Smooth1D/smooth1d}

\section{Numerical Experiments in Two Dimensions} \label{sec:experiments2D}
In this section we consider the model problem in $\Omega = (0,1)^2$:
\begin{subequations}\label{eq:model-2D}
\begin{align}
\nabla\cdot\left( a_\varepsilon(x) \nabla u \right) + \omega^2 u &=
b(x), &x&\in\Omega, \\
u(x) &= 0, &x&\in\partial\Omega.
\end{align}
\end{subequations}
Here, the forcing function $b$ will be taken to be an approximate
point source centered at $x_0 = (0.3, 0.4)$ given by:
\[ b(x) = \tfrac{s^2}{\pi} e^{-s^2 \|x - x_0\|^2}, \quad s = 20. \]
For both problems in two dimensions, the exact homogenized
coefficient is known. We compare the solution produced by WHMM to a
well resolved finite difference solution ($M_x = M_y =1000$) using
the exact coefficient which we refer to as the exact solution
$\bar{\vec{u}}$. In both examples, we also visualizes the solution to
\eqref{eq:model-2D} computed by direct numerical simulation for a
moderate $\varepsilon$. The experiments in two dimension match the
one dimensional results. Notably, we observe $O(H^2)$ convergence and accurate approximation of the homogenized coefficient $\bar{a}$.

\input{experiments/Periodic2D/periodic2d}

\input{experiments/Smooth2D/smooth2d}

\appendix
\input{Appendices/LocalAverageKernel}

\bibliographystyle{plain}
\bibliography{Bibliography/Bibliography}

\end{document}

%% file: experiments/Periodic1D/periodic1d.tex
\subsection{1D Periodic Coefficient} \label{sec:P1}
Consider problem \eqref{eq:model-1D} with coefficient
\[ a_\varepsilon(x) = \frac{11}{10} + \sin\frac{2\pi x}{\varepsilon}. \]
The homogenized coefficient for this problem is known to be the constant
\[
  \bar{a} = \sqrt{21/100}.
\]
We first illustrate that the kernels described in Section \ref{sec:micro-solver} can be used to  find accurate approximations to this value. Here, since the coefficient is periodic in $x/\varepsilon$, we use the $K^{(1, 7)}$ kernel based on the discussion in Section \ref{sec:num-kernel}. Additionally, we take $\eta = \tau = 4\varepsilon$ and $h = \varepsilon/32$ for the micro-problem (see Section \ref{sec:micro-solver}). 
Since the coefficient is periodic and by choice of $\eta, \tau,$ and $h$, the micro-problem is independent of $\varepsilon$, so we take $\varepsilon = 1$ when computing $\alpha$ offline as discussed in Section \ref{sec:offline-coefficient}.
We observe
\[
  \frac{|\alpha - \bar{a}|}{\bar{a}} \sim 1.31 \times 10^{-5}.
\]
On the other hand, if we compute $\alpha$ as a simple average (i.e. use $K^{(1,-1)}$), the relative error is $6.9\times 10^{-2}$.
We now illustrate how the accuracy in $\alpha$ affects the WHMM solution. First consider the case when it is computed very accurately. Then because the coefficient is periodic in $x / \varepsilon$, Theorem \ref{thm:main_theorem} applies directly to this problem, and we expect the relative error between the exact homogenized solution $\bar{\vec{u}}$ and the WHMM solution $\vec{v}_h$ to be $O(H^2\omega^3 / \delta)$,
or $O(H^2)$ for fixed $\omega$ and $\delta$. In the left panel of Figure~\ref{fig:P1_error_v_H}, we plot the relative error $\|\bar{\vec{u}} - \vec{v}_h\|_H / \|\bar{\vec{u}}\|_H$ for 50 logarithmically spaced values of $H$ between $10^{-3}$ and $10^{-2}$, the second-order convergence rate is clearly visible.
Now consider the case when $\alpha$ is approximated using a simple average, then repeating the error computation for $H$ between $10^{-3}$ and $10^{-2}$ results in the errors observed in the right panel of Figure \ref{fig:P1_error_v_H}. As can be seen, the poor approximation of the homogenized coefficient as a simple average of the flux completely destroys the convergence rate.

\begin{figure}[hbt]
  \centering
  \includegraphics[height=1.5in]{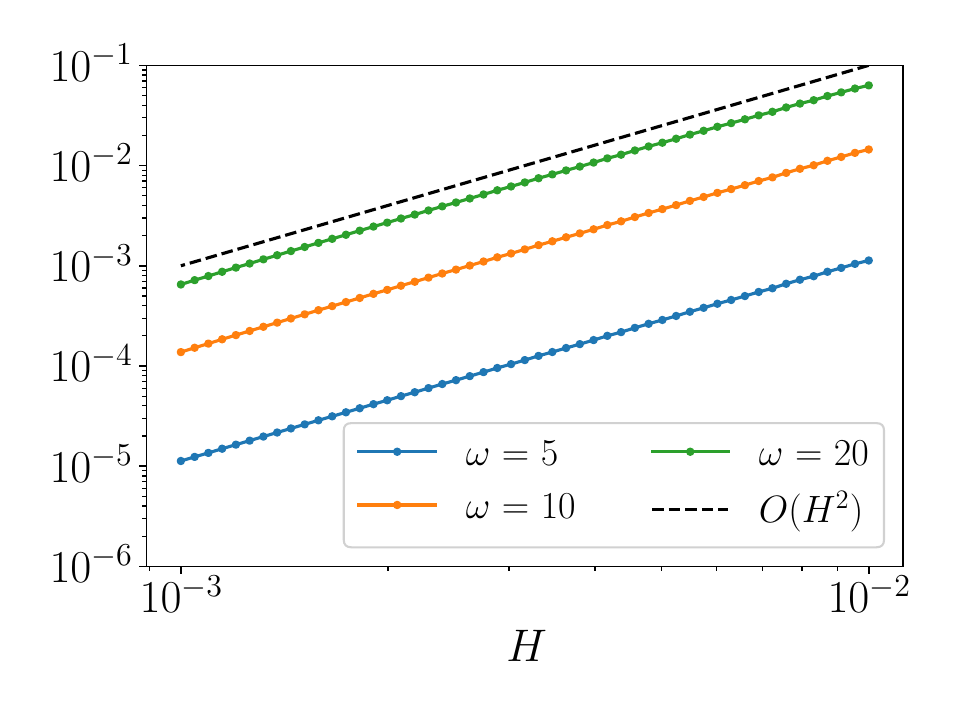}
  \includegraphics[height=1.5in]{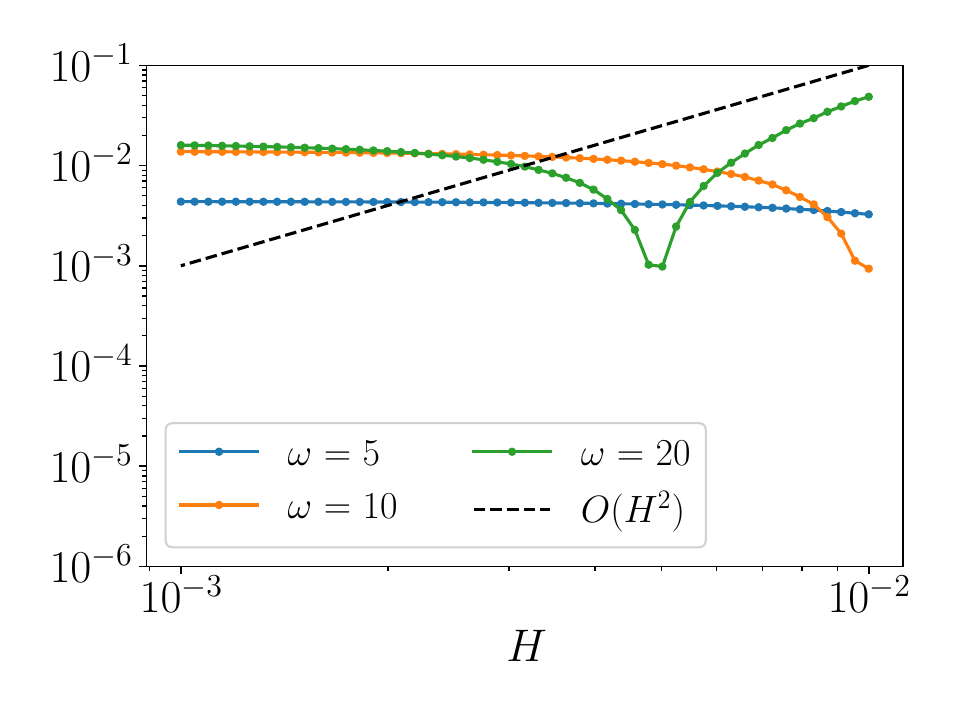}
  \caption{Displayed is the relative error between the WHMM solution and the exact homogenized ($y$-axis) solution as a function of the macro grid size $H$ for various frequencies. The second order rate of convergence is clearly displayed on the left where the $\bar{a}$ is approximated accurately, and no convergence on the right where $\bar{a}$ is approximated poorly.}
  \label{fig:P1_error_v_H}
\end{figure}

\begin{figure}[hbt]
  \centering
  \includegraphics[height=1.5in]{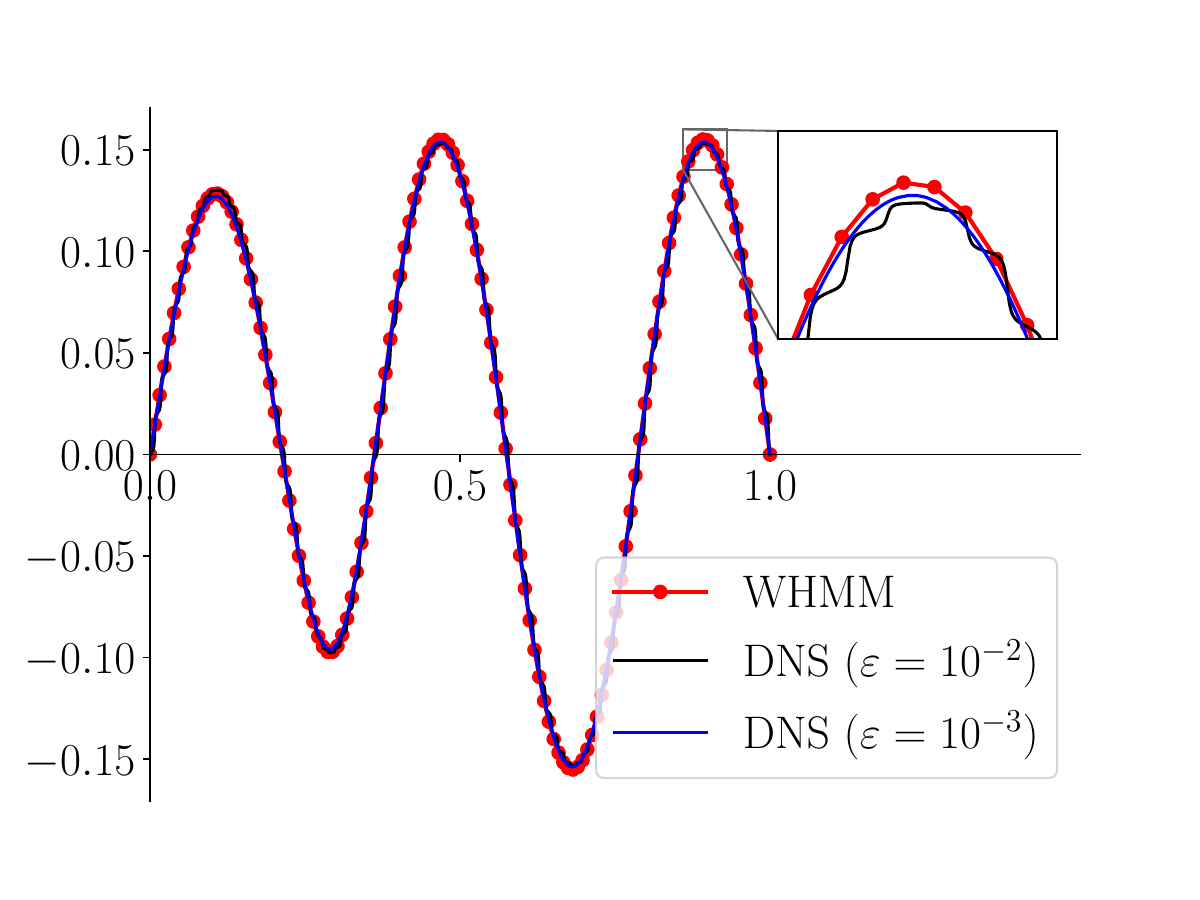}
  \includegraphics[height=1.5in]{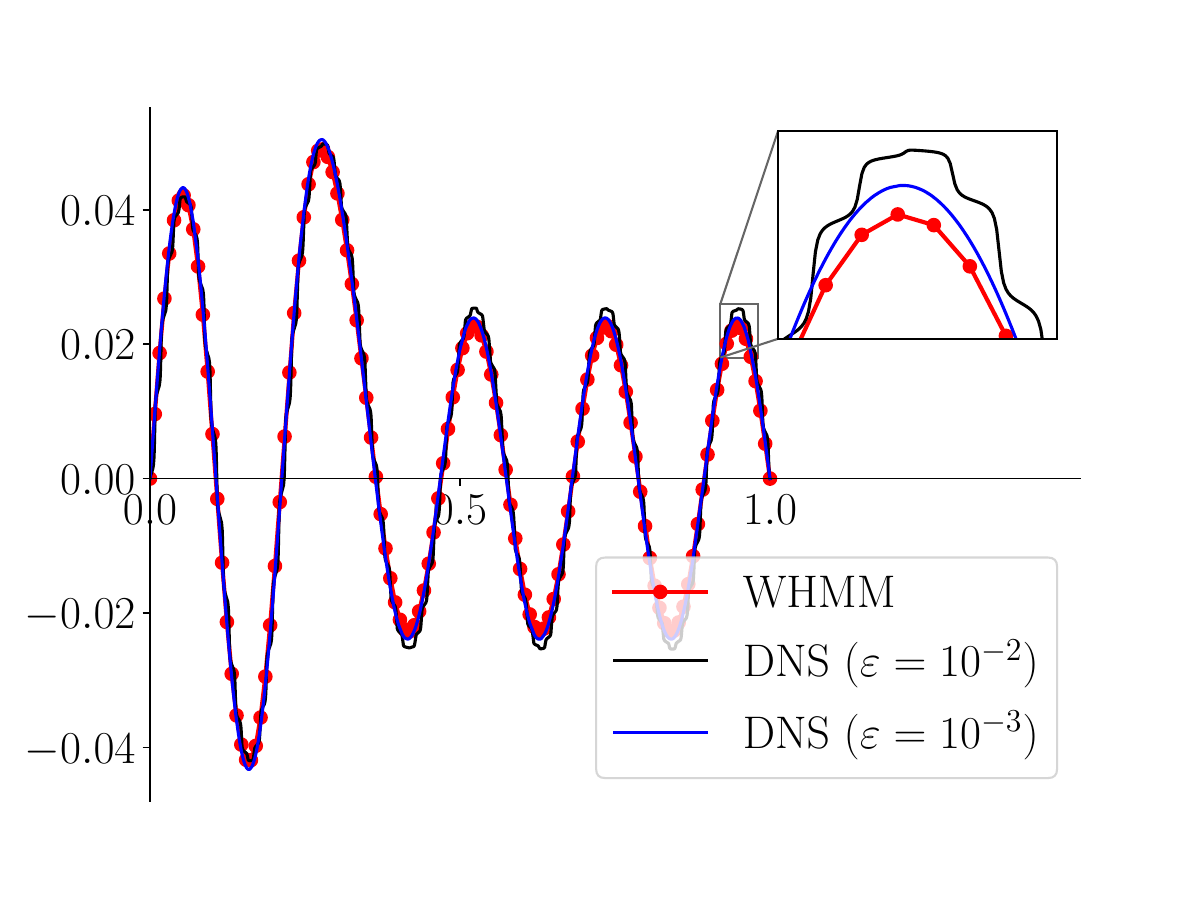}
  \caption{The DNS solution and the WHMM solution to the problem described in Section \ref{sec:P1} with $\omega = 10$ (left) and $\omega = 20$ (right) for $\varepsilon = 10^{-2}$ and $\varepsilon = 10^{-3}$.}
  \label{fig:P1_solution}
\end{figure}

Next, we compare the WHMM solution $\vec{v}_h$  to the DNS solution $\vec{u}_\varepsilon$ and the exact solution $\bar{\vec{u}}$.
In this experiment we use the $K^{(1,7)}$ kernel so that $\alpha$ is accurate enough to ensure that the WHMM solver produces a $\bar{\vec{v}}_h$
that is essentially identical to the solution to the homogenized problem $\bar{\vec{u}}_h$ on the same grid.

Two example DNS solutions with $\omega = 10$ and $\omega = 20$ for the moderate $\varepsilon = 10^{-2}$ are plotted in Figure \ref{fig:P1_solution} along corresponding homogenized solutions (computed with WHMM) at a fixed $M_x = 128$. The relative difference
$\|I_H\vec{v}_h - \vec{u}_\varepsilon\|_{L^2} / \|\vec{u}_\varepsilon\|_{L^2}$ are 0.060 and  0.108 for $\omega = 10$ and 20, respectively. We then repeat the computations with $\varepsilon = 10^{-3}$. The results are plotted in the same figure. The relative errors are now 0.010 and 0.031, for $\omega = 10$ and 20, respectively.

To the left in Figure \ref{fig:P1_error_v_omega}, we plot the relative errors $\|I_H\vec{v}_h - \vec{u}_\varepsilon\|_{L^2} / \|\vec{u}_\varepsilon\|_{L^2}$ and $\|\vec{v}_h - \bar{\vec{u}}\|_H / \|\bar{\vec{u}}\|_H$ as a function of $\omega$. Here we take $M_x$ according to \eqref{eq:ppw}. 
We would then expect
a relative error of size $O(\varepsilon) + O(1/\delta)$,
which corresponds to the
homogenization error and a fixed
discretization error.
We see that the error indeed appears to spike at certain values of $\omega$. These values of $\omega$ are close to eigenvalues of the elliptic operator $\frac{d}{dx}(\bar{a} \frac{d}{dx})$ implying $\delta$, the relative distance to resonance, is very small. When $\omega$ is exactly an eigenvalue, the Helmholtz problem is singular, so these spikes in the error reflect the singularity and are expected. Note also that we have marked errors for the solutions displayed in Figure \ref{fig:P1_solution} by squares
in Figure \ref{fig:P1_error_v_omega}, left.

To the right in Figure \ref{fig:P1_error_v_omega} we plot the relative errors as a function of $\varepsilon$ for fixed $\omega$ and $\delta$.
Here we see that for larger values of $\varepsilon$ the error is dominated by the $O(\varepsilon)$ homogenization error, and for very small values of $\varepsilon$ the error is dominated by the $O(1/\delta)$ discretization error. The colored dashed lines represent the error between $\vec{v}_h$ and $\bar{\vec{u}}$ which explains the plateau observed in the errors, as $\vec{v}_h \to \bar{\vec{u}}$ rather than $\vec{u}_{\varepsilon}$, further validating the asymptotic behavior of the method.

\begin{figure}[htb]
  \centering
  \includegraphics[height=1.5in]{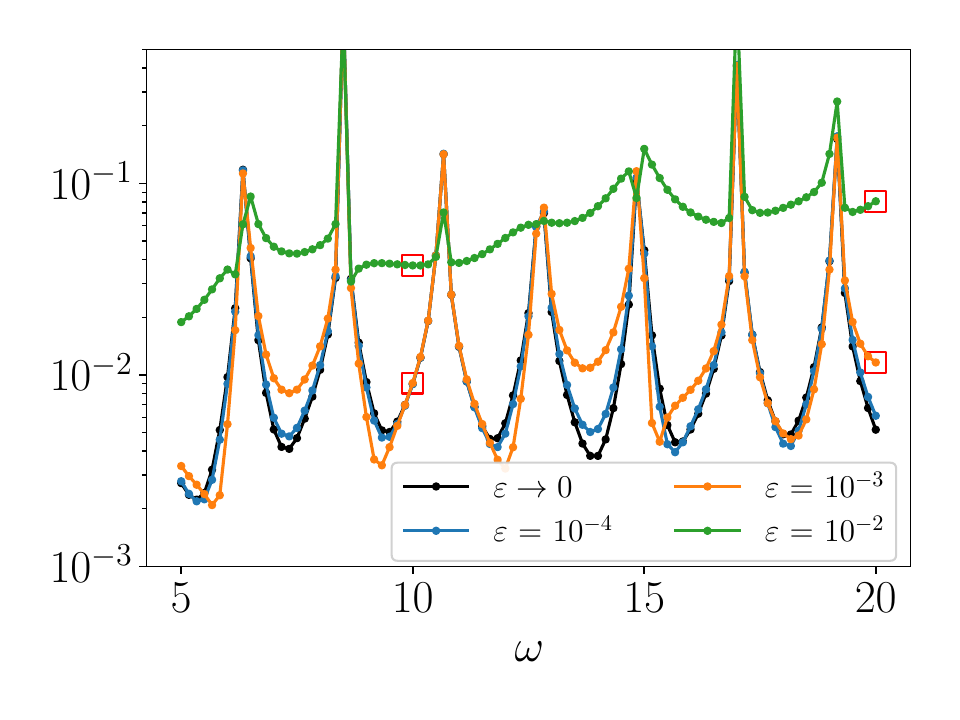}
  \includegraphics[height=1.5in]{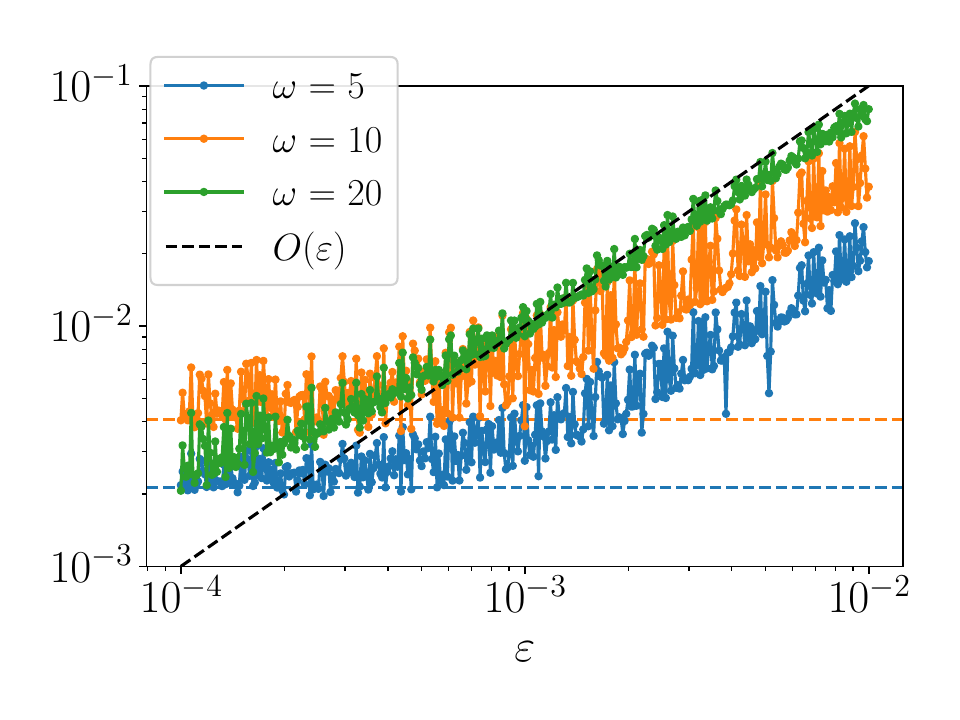}
  \caption{
    (left) The relative error between the WHMM solution and the DNS solution as a function of $\omega$ for various epsilon as well as the exact homogenized solution $\bar{\vec{u}}$ (represented in the graph by $\varepsilon \to 0$).
    (right) The relative error between the WHMM solution and the DNS solution as a function of $\varepsilon$ for various $\omega$. The horizontal dashed lines represent the relative error between WHMM and the exact homogenized solution. Note that there is no correlation between the colors used in the figures.
  }
  \label{fig:P1_error_v_omega}
\end{figure}

%% file: experiments/Smooth1D/smooth1d.tex
\subsection{1D Locally Periodic Coefficient}\label{sec:P2}

Consider problem \eqref{eq:model-1D} with coefficient
\[ a_\varepsilon(x) = \frac{11}{10} + \frac{1}{2}\sin 2\pi x + \frac{1}{2}\sin\frac{2\pi x}{\varepsilon}. \]
For this problem, the homogenized coefficient is known and is non-constant:
\[ \bar{a}(x) = \sqrt{\left( \frac{11}{10} + \frac{1}{2}\sin 2\pi x \right)^2 - \frac{1}{4}}. \]
Here, the coefficient is periodic in $x/\varepsilon$ but also depends on $x$, so for this problem we use the $K^{(5, 7)}$ kernel and take $\eta = \tau = 8\varepsilon$ and $h = \varepsilon / 32$ for the micro-problem. With this choice, the maximum error observed across all $\varepsilon$ ranging from $10^{-4}$ to $10^{-2}$ was
\[
  \max_{\varepsilon \in[10^{-4},10^{-2}]}\frac{\max_j |\alpha_j - \bar{a}_j|}{\max_j \bar{a}_j} \sim 5.6 \times 10^{-6}.
\]
The homogenized coefficient $\alpha$ computed by HMM as well as the pointwise error are plotted in Figure \ref{fig:P2_hom_coeff}.
Two example solutions are plotted in Figure \ref{fig:P2_solution}, the relative errors $\|I_H \vec{v}_h - \vec{u}_\varepsilon\|_{L^2}/\|\vec{u}_\varepsilon\|_{L^2}$ 
are $0.016$ and $0.028$ for $\omega = 10$ and $20$, respectively.

\begin{figure}[h]
  \centering
  \includegraphics[height=1.5in]{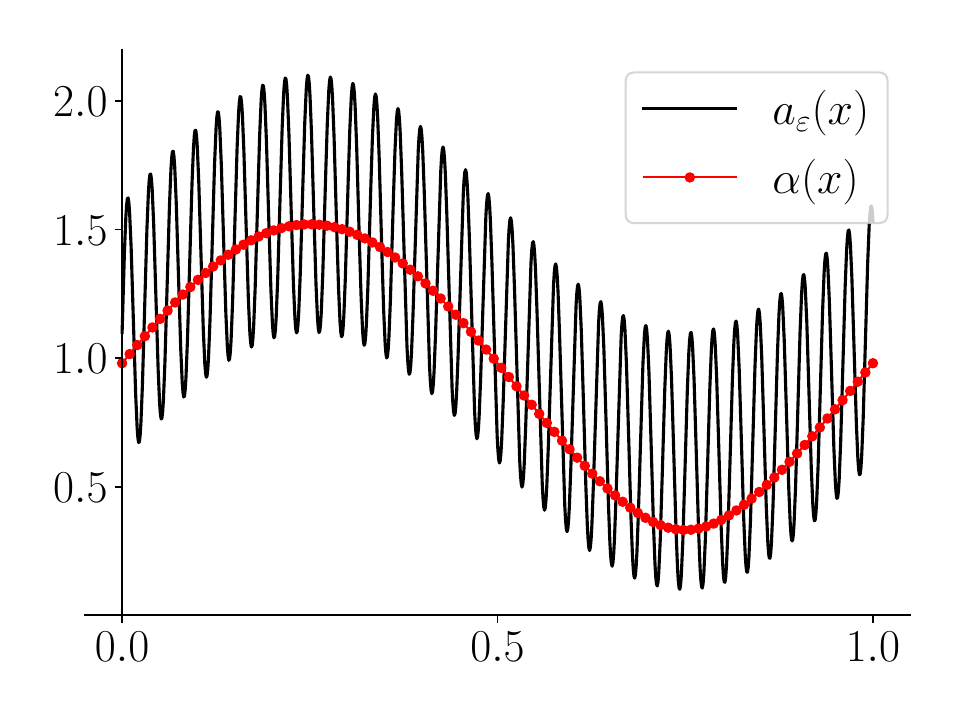}
  \includegraphics[height=1.5in]{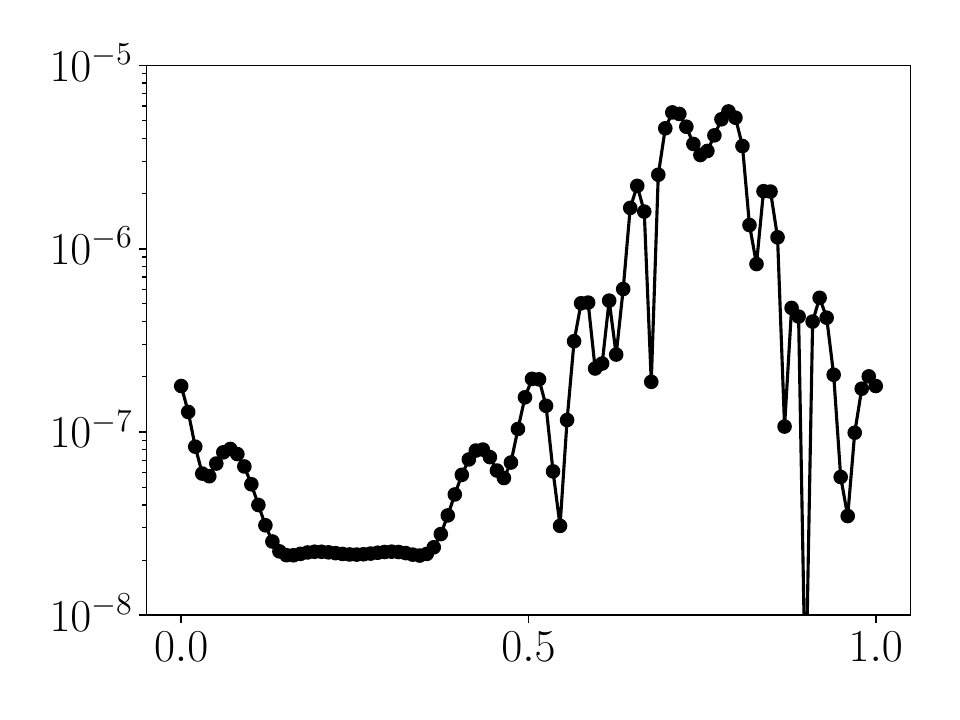}
  \caption{
    (left) The high frequency coefficient $a_\varepsilon(x)$ ($\varepsilon = 0.03$) and the homogenized coefficient computed by HMM.
    (right) The relative error $|\bar{a} - \alpha| / \max|\bar{a}|.$
  }
  \label{fig:P2_hom_coeff}
\end{figure}

\begin{figure}[h]
  \centering
  \includegraphics[height=1.5in]{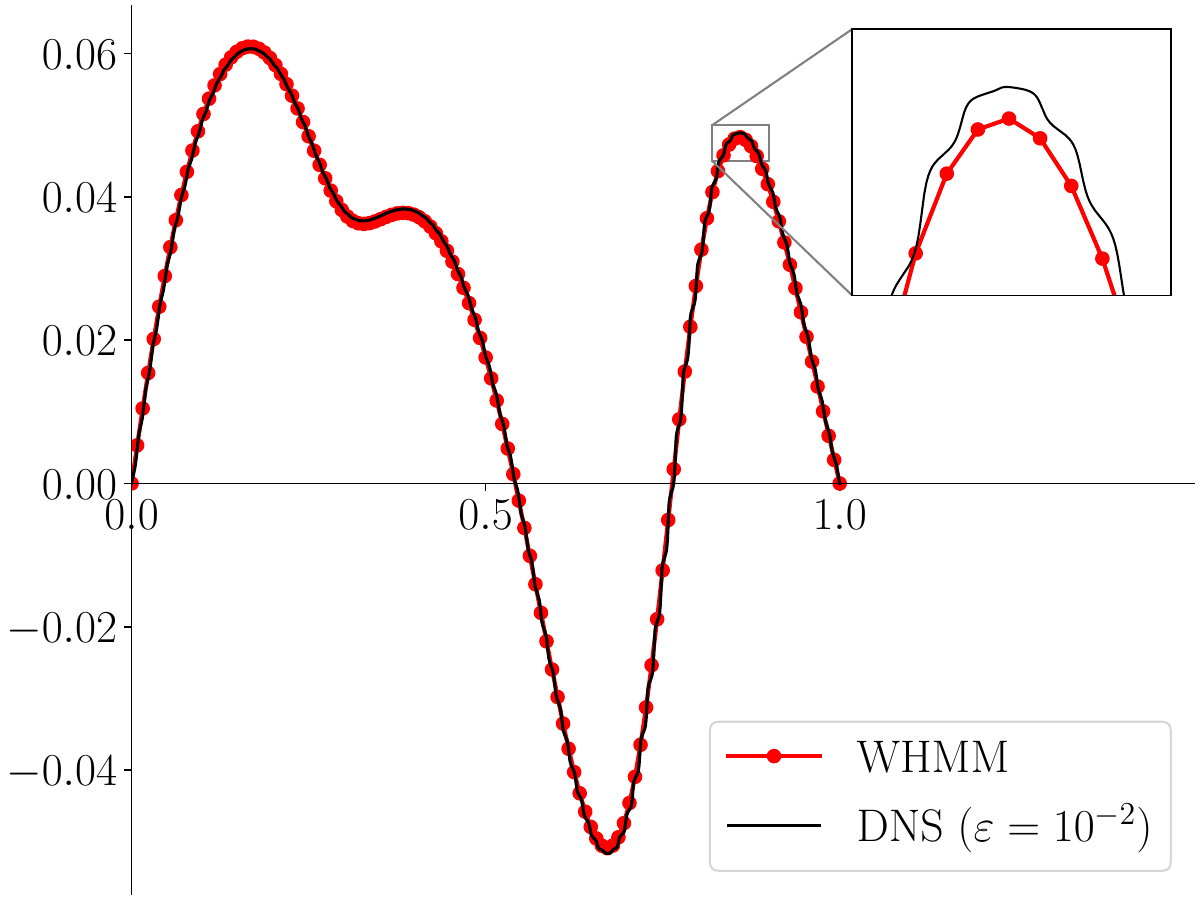}
  \includegraphics[height=1.5in]{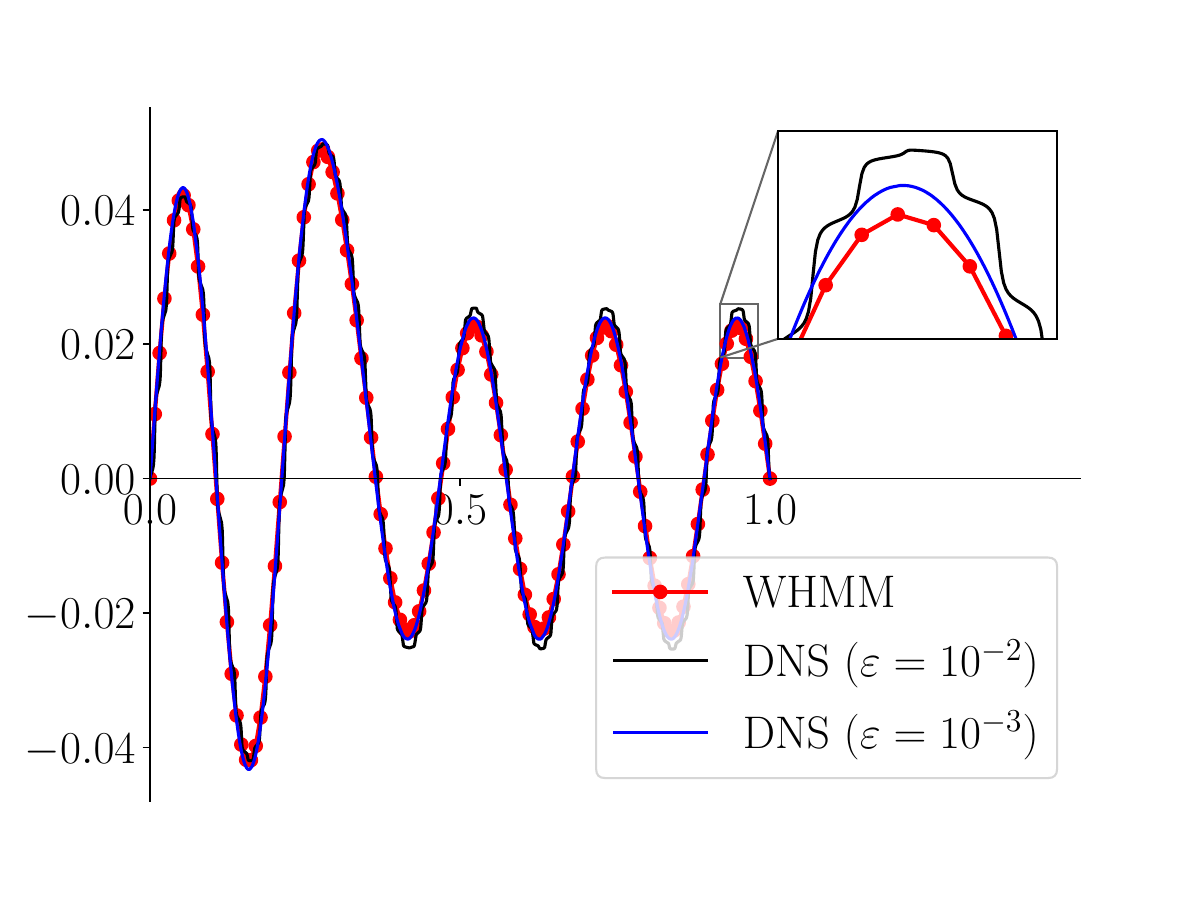}
  \caption{The DNS solution with $\varepsilon = 10^{-2}$ and the WHMM solution to the problem in Section \ref{sec:P2} with $\omega = 10$ and $\omega = 20$, respectively.}
  \label{fig:P2_solution}
\end{figure}

For this problem, the assumptions of Theorem \ref{thm:main_theorem} do not apply. Nevertheless, previous results for the HMM method for the time dependent wave equation suggest HMM generalizes well to problem with locally periodic coefficients. We repeat the experiment in Section \ref{sec:P1} for this coefficient.
In the left panel of Figure \ref{fig:P2_error_v_omega} we once again observe an $O(H^2)$ convergence of the WHMM solution to the exact homogenized solution.
In the middle and right panels of Figure \ref{fig:P2_error_v_omega}, we see the same pattern as in \ref{sec:P1}. The error spikes for discrete values of $\omega$ where $\delta \to 0$ and the problem is ill-posed. The error between $\mathbf{v}_h$ and $\mathbf{u}_\varepsilon$ is roughly $O(\varepsilon)$ for larger values of $\varepsilon$ and $O(1/\delta)$ for smaller values. Note the red boxes marked in the middle panel of Figure \ref{fig:P2_error_v_omega} correspond to the solutions in Figure \ref{fig:P2_solution}.


\begin{figure}[h]
  \centering
  \includegraphics[height=1.25in]{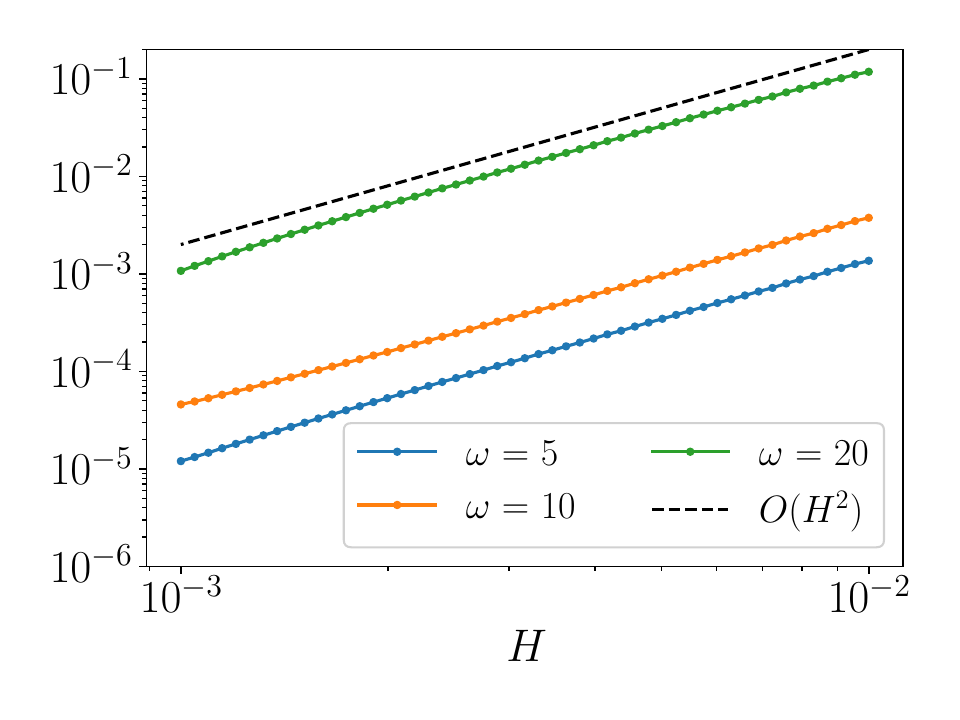}
  \includegraphics[height=1.25in]{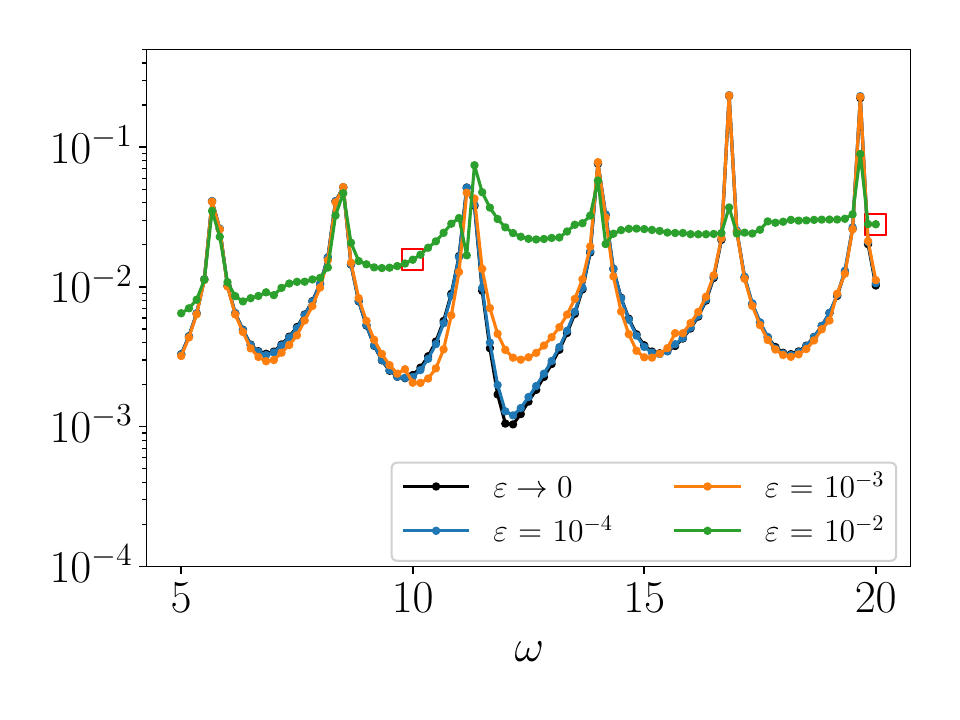}
  \includegraphics[height=1.25in]{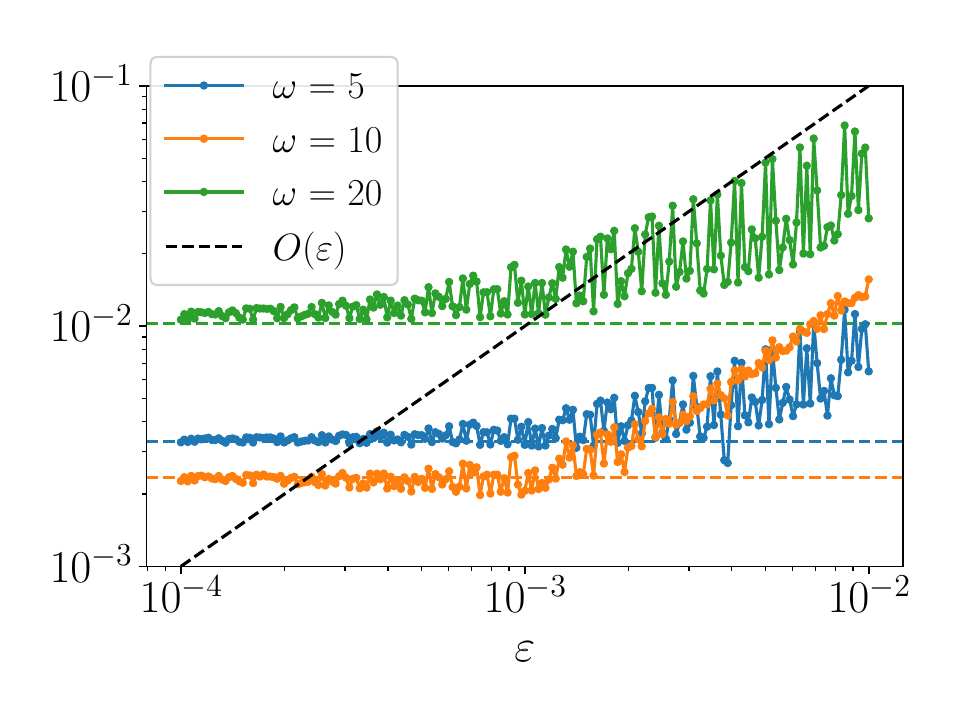}
  \caption{
    (left) The relative error between the WHMM solution and the exact homogenized solution as a function of $H$ for various frequencies. (center) The relative error between $\mathbf{v}_h$ and $\mathbf{u}_\varepsilon$ as a function of $\omega$ for various $\varepsilon$ as well as the exact solution $\bar{u}$ ($\varepsilon \to 0$).
    (right) The relative error between $\mathbf{v}_h$ and $\mathbf{u}_\varepsilon$ as a function of $\varepsilon$ for various $\omega$. The horizontal dashed lines represent the relative error between $\mathbf{v}_h$ and $\vec{\bar{u}}$.
  }
  \label{fig:P2_error_v_omega}
\end{figure}

%% file: experiments/Periodic2D/periodic2d.tex
\subsection{2D Periodic Coefficient} \label{sec:P4}
Consider problem \eqref{eq:model-2D} with coefficient
\[ a_\varepsilon(x) = \frac{11}{10} + \sin 2\pi \frac{x_1}{\varepsilon}. \]
For this problem, the homogenized coefficient is known to be the constant matrix:
\[ \bar{a} =
  \begin{pmatrix}
    \sqrt{0.21} & 0 \\
    0 & 1.1
\end{pmatrix}. \]
Two example solutions are plotted in Figure \ref{fig:P4_solution}. Since the coefficient is periodic, we use the $K^{(1,7)}$ kernel and take $\eta = \tau = 8\varepsilon$ with $h = \varepsilon / 32$ for the micro-problem. With this choice, we observe
\[ \frac{\|\alpha - \bar{a}\|_F}{\|\bar{a}\|_F} \sim 4.9 \times 10^{-7}. \]
In Table \ref{table:p4} we report the relative error of the WHMM solution $\vec{v}_h$ compared to the exact homogenized solution $\bar{\vec{u}}$ for various values of $\omega$ and various grid sizes. We roughly observe the predicted $O(H^2)$ convergence which was proven for one dimensional problems, but not in two dimensions.

\begin{figure}
  \centering
  \includegraphics[height=1.4in]{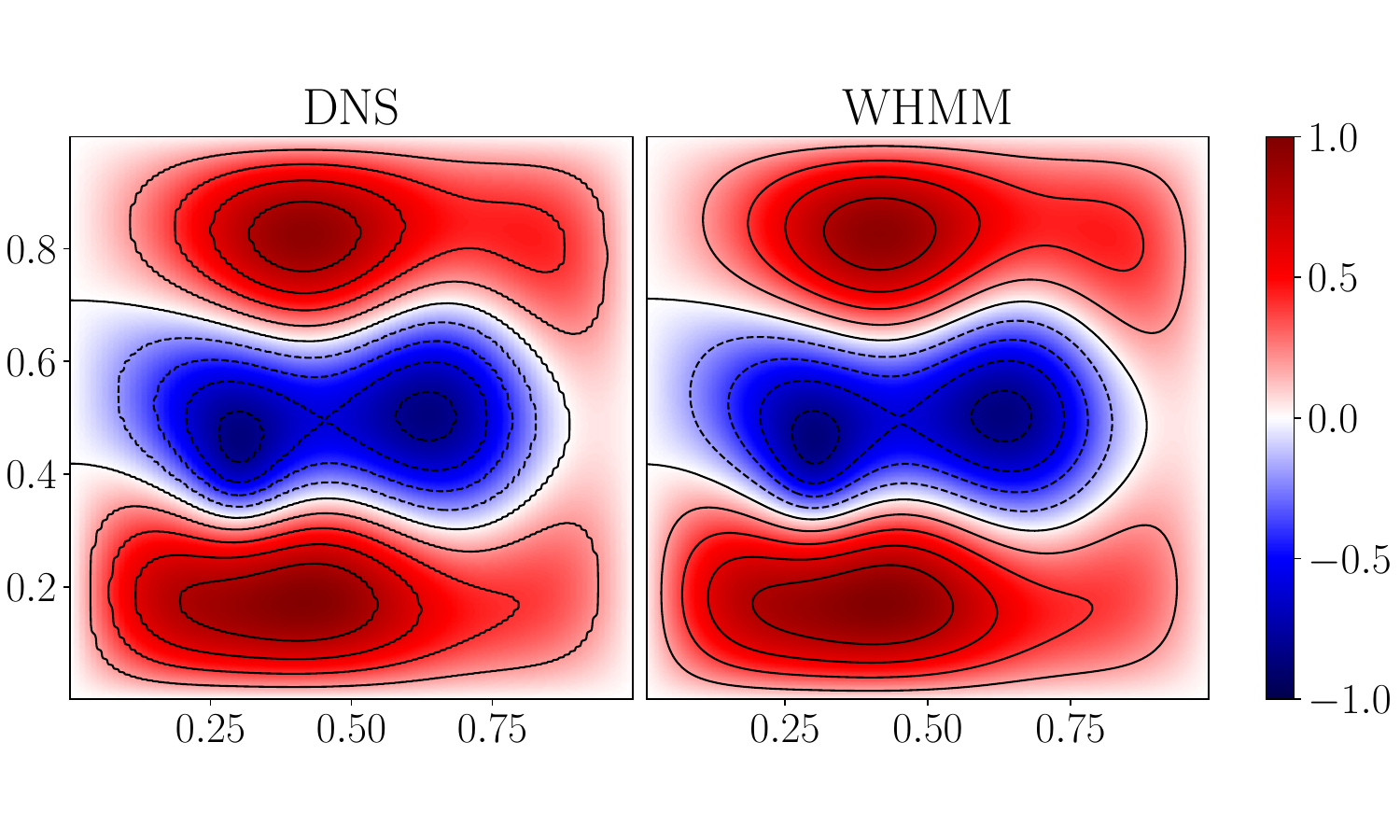}
  \includegraphics[height=1.4in]{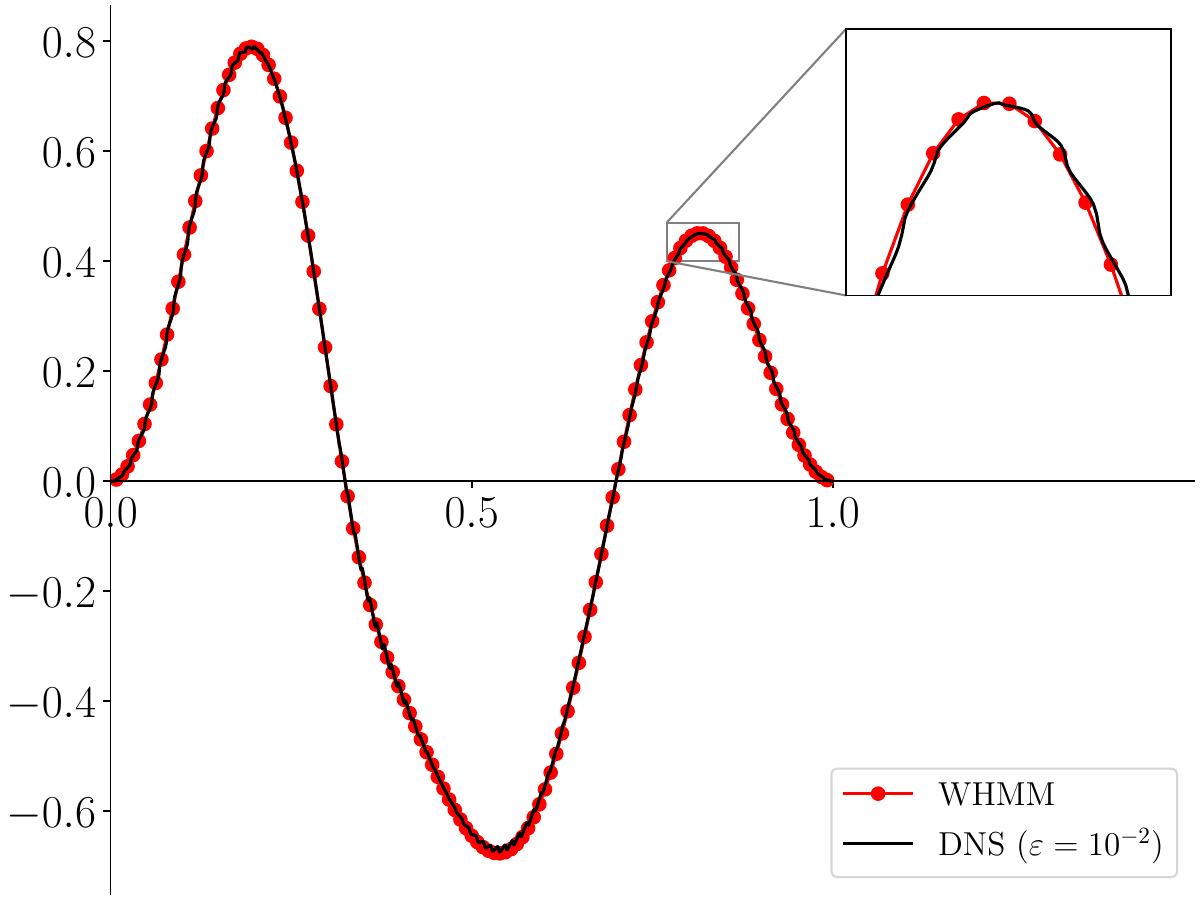} \\
  \includegraphics[height=1.4in]{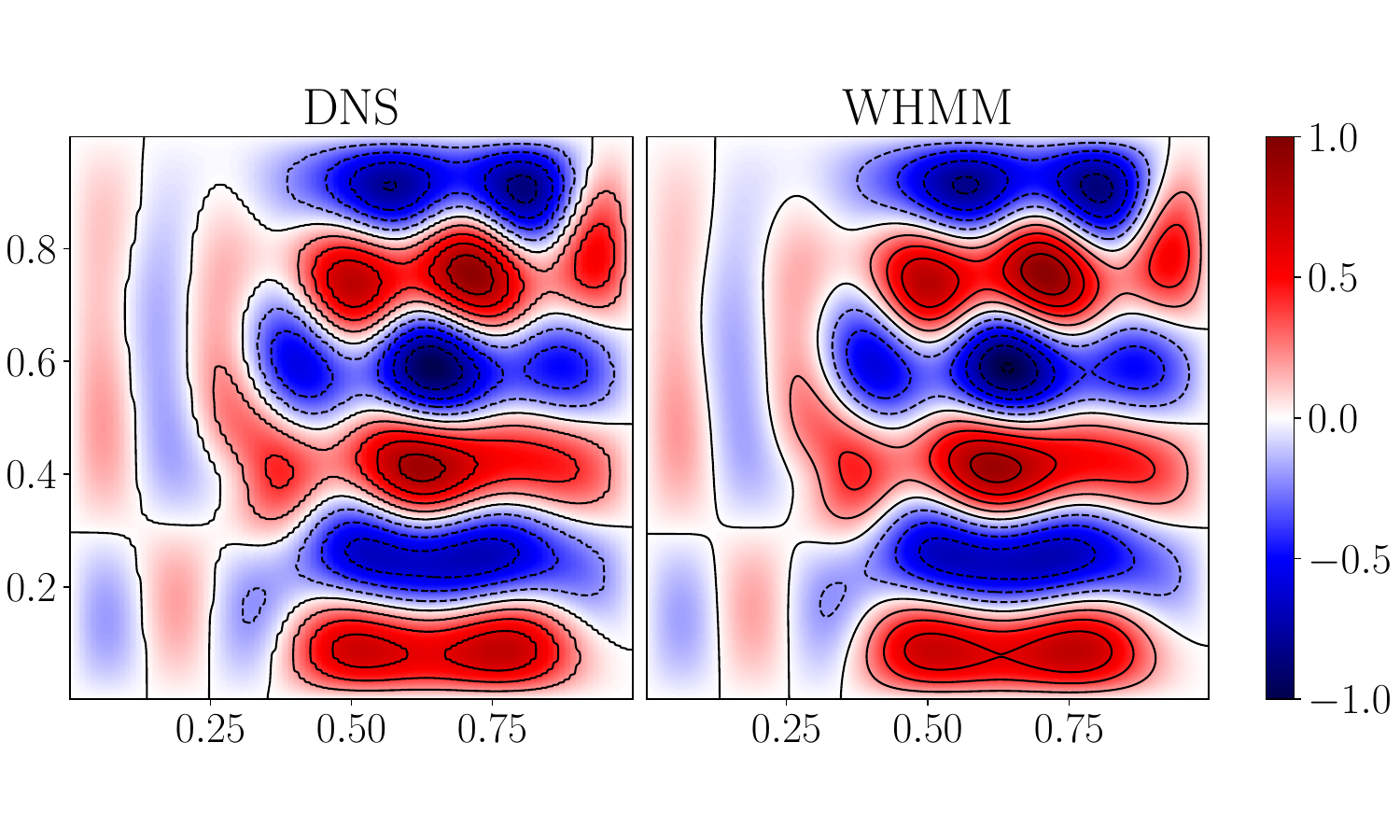}
  \includegraphics[height=1.4in]{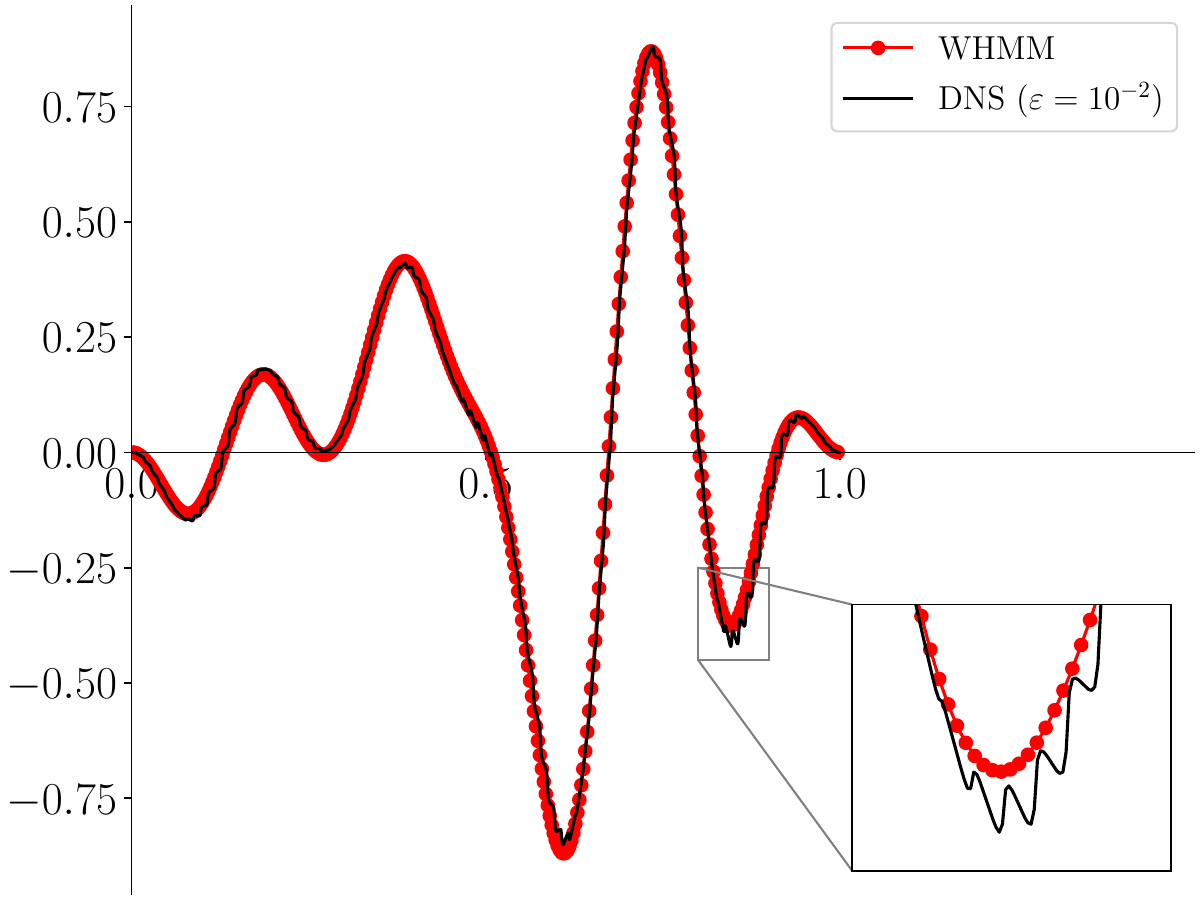} \\
  \caption{
    (top left) Contour plot of the solution for $\omega = 10$.
    (top right) The solution along $y = x$ for $\omega = 10$.
    (bottom left) Contour plot of the solution for $\omega = 20$.
    (bottom right) The solution along $y = x$ for $\omega = 20$.
  }
  \label{fig:P4_solution}
\end{figure}

\input{experiments/Periodic2D/P4_table}

%% file: experiments/Periodic2D/P4_table.tex
\begin{table}[h!]
  \centering
  \small
  \setlength{\tabcolsep}{4pt}
  \begin{tabular}{c|cc|cc|cc|}
    $M_x = M_y$ & \multicolumn{2}{c|}{${\omega} = 5.0$} & \multicolumn{2}{c|}{${\omega} = 10.0$} & \multicolumn{2}{c|}{${\omega} = 20.0$} \\
    \hline
    & rel. error & rate & rel. error & rate & rel. error & rate \\
    \hline
    50 & $2.68 \times 10^{-2}$ & $-$ & $1.76 \times 10^{-1}$ & $-$ & $7.97 \times 10^{-1}$ & $-$ \\
    100 & $8.12 \times 10^{-3}$ & $1.7$ & $4.49 \times 10^{-2}$ & $2.0$ & $2.54 \times 10^{-1}$ & $1.6$ \\
    200 & $2.44 \times 10^{-3}$ & $1.7$ & $1.12 \times 10^{-2}$ & $2.0$ & $7.24 \times 10^{-2}$ & $1.8$ \\
    400 & $9.16 \times 10^{-4}$ & $1.4$ & $2.75 \times 10^{-3}$ & $2.0$ & $1.73 \times 10^{-2}$ & $2.1$ \\

  \end{tabular}
  \caption{The relative error between the exact solution $\Bar{\vec{u}}$ and the WHMM solution $\vec{v}_h$.}
  \label{table:p4}
\end{table}

%% file: experiments/Smooth2D/smooth2d.tex
\subsection{2D Locally Periodic Coefficient}\label{sec:P5}

Consider problem \eqref{eq:model-2D} with coefficient
\[ a_\varepsilon(x) =  \frac{11}{10} + \frac{1}{2}\sin 2\pi x_1 + \frac{1}{2}\sin 2\pi \frac{x_1}{\varepsilon}. \]
For this problem, the homogenized coefficient is known to be the matrix-valued function:
\[ \bar{a}(x) =
  \begin{pmatrix}
    \sqrt{g(x)^2 - \frac{1}{4}} & 0 \\
    0 & g(x)
\end{pmatrix}, \qquad g(x) = \frac{11}{10} + \frac{1}{2}\sin 2\pi x_1. \]
Two example solutions are plotted in Figure \ref{fig:P5_solution}. Now, we use the $K^{(5,7)}$ kernel and take $\eta = \tau = 8\varepsilon$ with $h = \varepsilon / 32$ for the micro-problem. With this choice, we observe
\[ \frac{\max_{i,j} \|\alpha_{i,j} - \bar{a}_{i,j}\|_F}{\max_{i,j}\|\bar{a}_{i,j}\|_F} \sim 6.14 \times 10^{-5}. \]
In Table \ref{table:p5} we report the relative error of the WHMM solution $\vec{v}_h$ compared to the exact homogenized solution $\bar{\vec{u}}$ for various values of $\omega$ and various grid sizes. We roughly observe $O(H^2)$ convergence.

\begin{figure}
  \centering
  \includegraphics[height=1.4in]{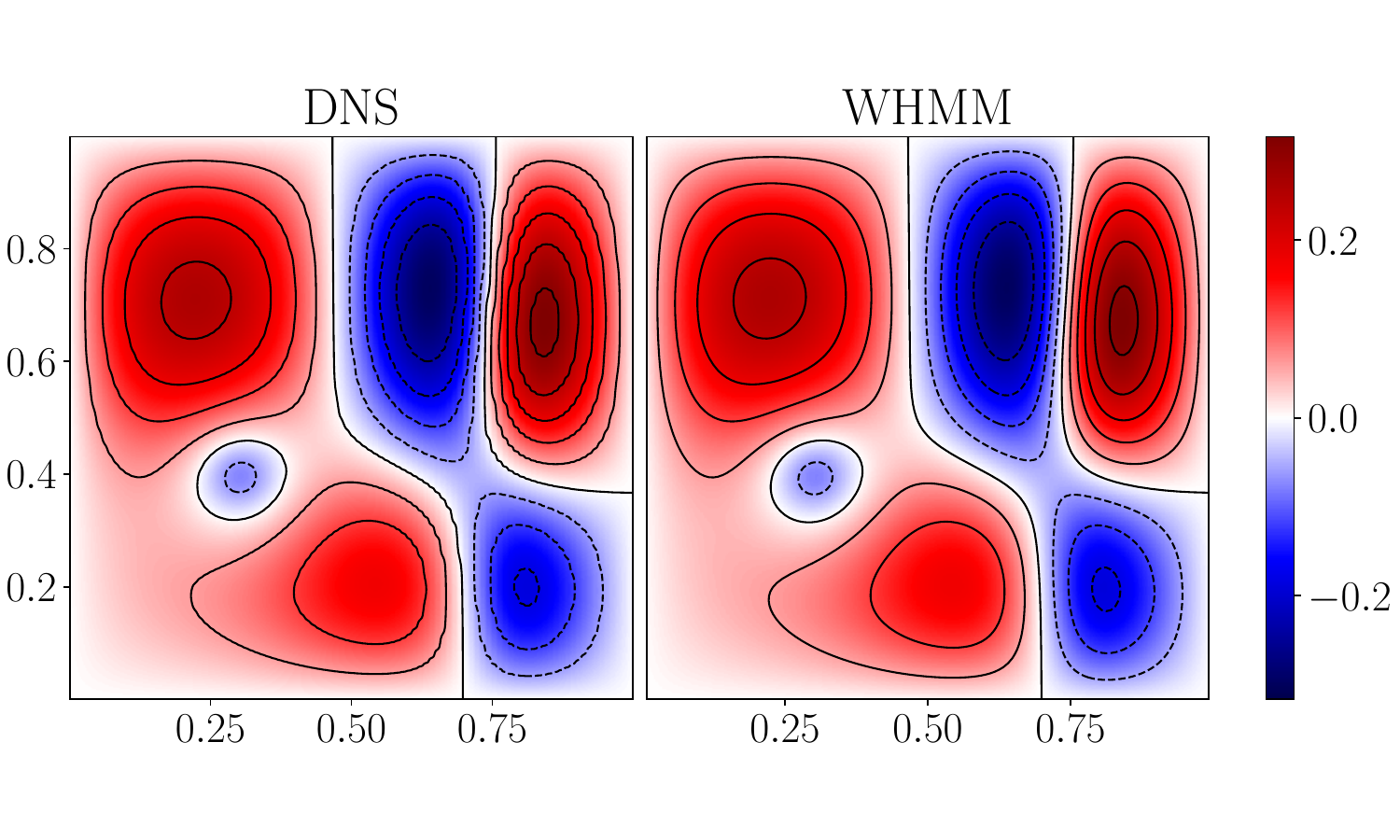}
  \includegraphics[height=1.4in]{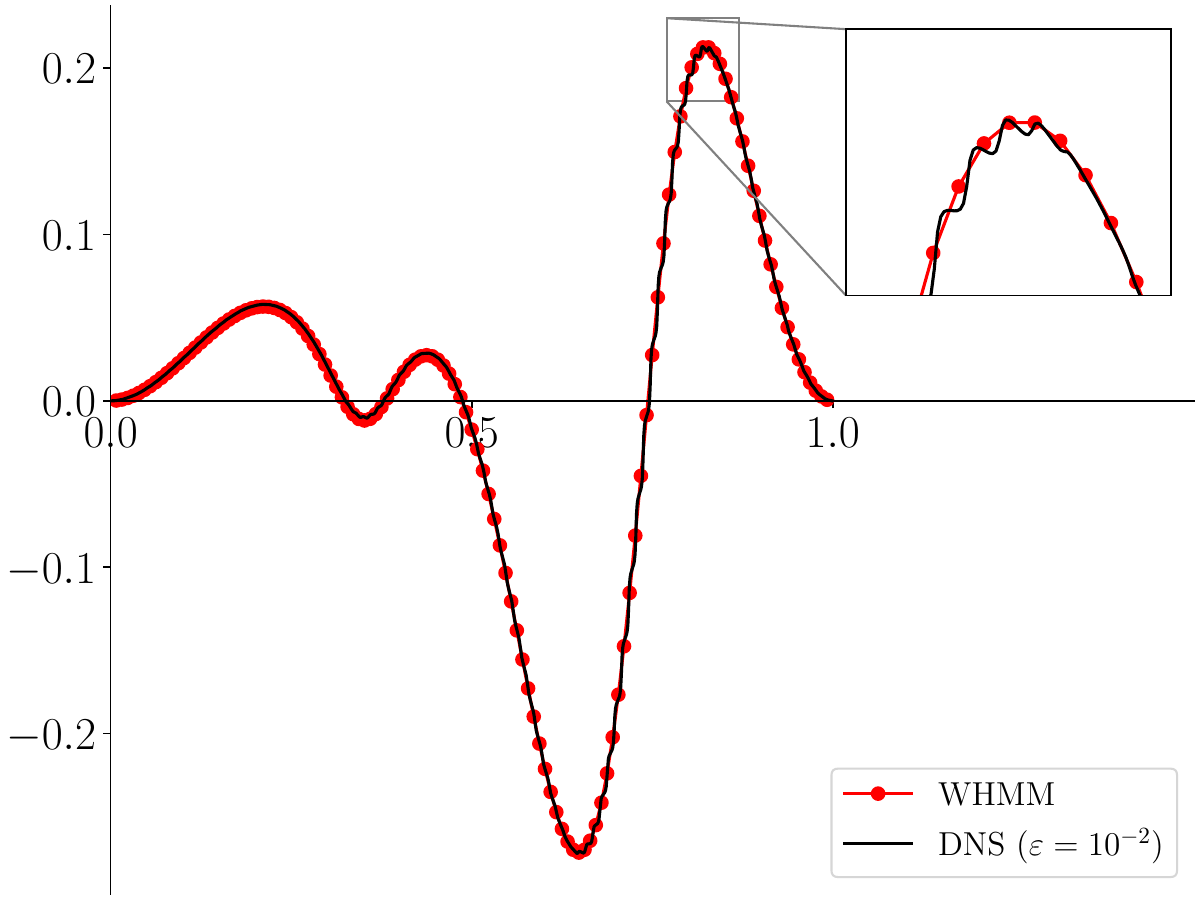} \\
  \includegraphics[height=1.4in]{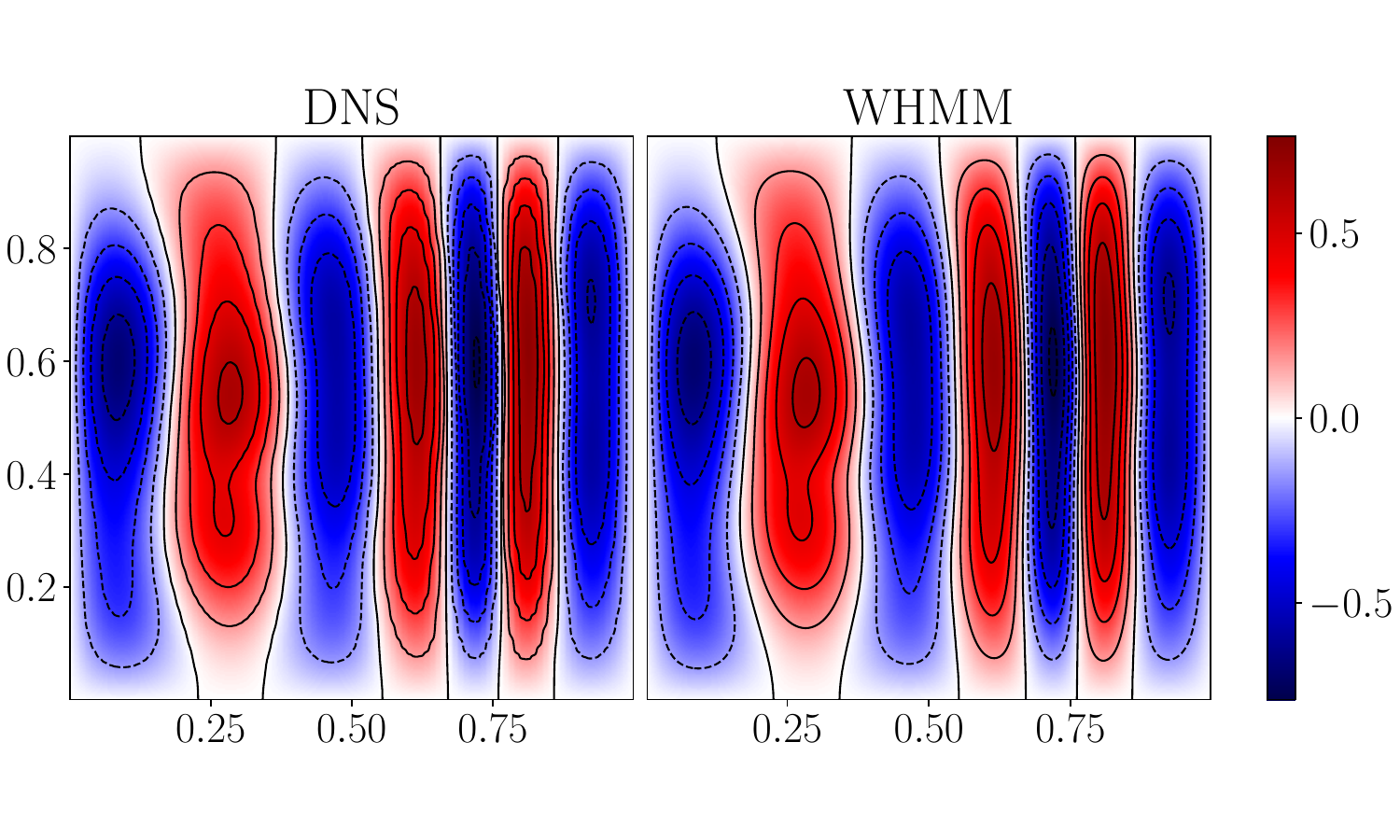}
  \includegraphics[height=1.4in]{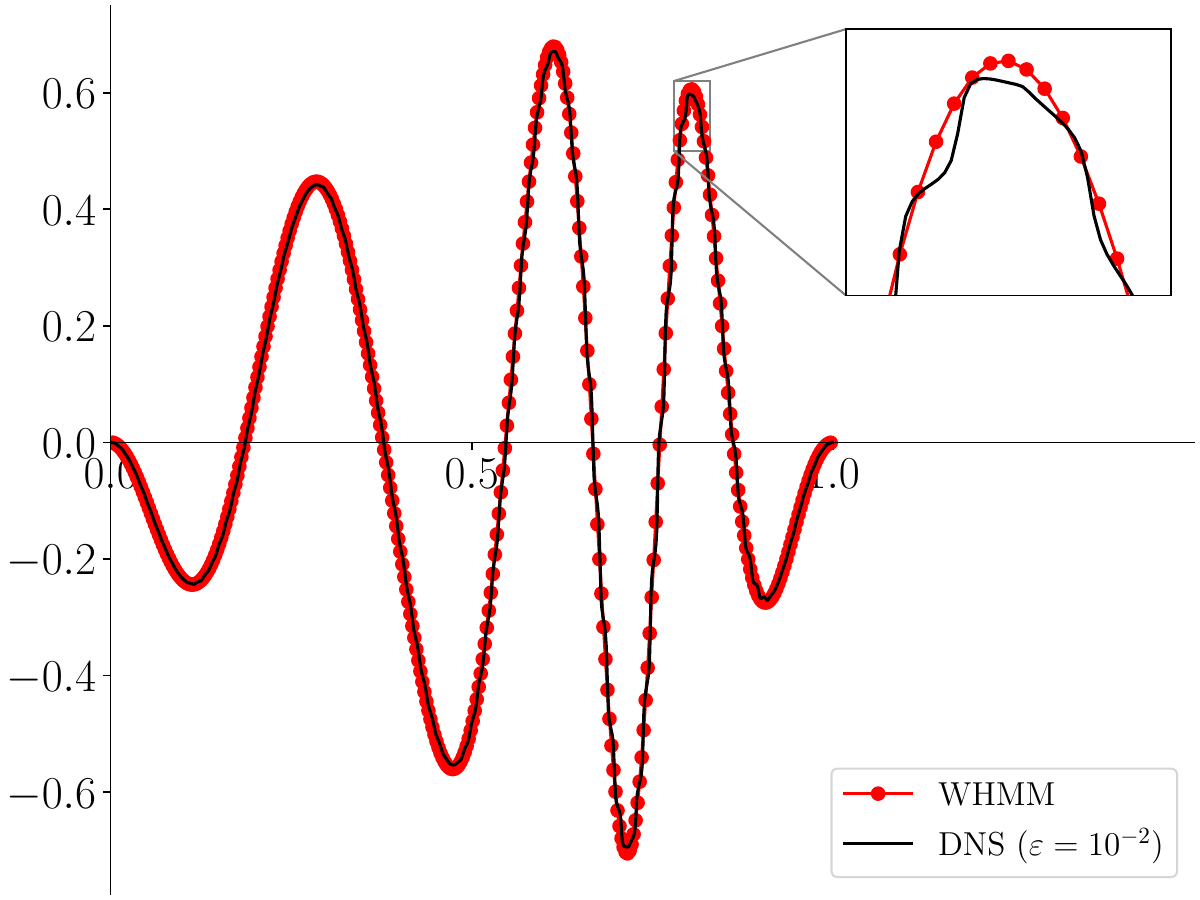} \\
  \caption{
    (top left) Contour plot of the solution for $\omega = 10$.
    (top right) The solution along $y = x$ for $\omega = 10$.
    (bottom left) Contour plot of the solution for $\omega = 20$.
    (bottom right) The solution along $y = x$ for $\omega = 20$.
  }
  \label{fig:P5_solution}
\end{figure}

\input{experiments/Smooth2D/P5_error_table}

%% file: experiments/Smooth2D/P5_error_table.tex
\begin{table}[h!]
  \centering
  \small
  \setlength{\tabcolsep}{4pt}
  \begin{tabular}{c|cc|cc|cc|}
    $M_x = M_y$ & \multicolumn{2}{c|}{${\omega} = 5.0$} & \multicolumn{2}{c|}{${\omega} = 10.0$} & \multicolumn{2}{c|}{${\omega} = 20.0$} \\
    \hline
    & rel. error & rate & rel. error & rate & rel. error & rate \\
    \hline
    50 & $1.51 \times 10^{-2}$ & $-$ & $6.52 \times 10^{-2}$ & $-$ & $6.87 \times 10^{-1}$ & $-$ \\
    100 & $5.06 \times 10^{-3}$ & $1.6$ & $1.79 \times 10^{-2}$ & $1.9$ & $3.55 \times 10^{-1}$ & $0.95$ \\
    200 & $1.68 \times 10^{-3}$ & $1.6$ & $5.14 \times 10^{-3}$ & $1.8$ & $1.21 \times 10^{-1}$ & $1.5$ \\
    400 & $5.41 \times 10^{-4}$ & $1.6$ & $1.49 \times 10^{-3}$ & $1.8$ & $3.27 \times 10^{-2}$ & $1.9$ \\
  \end{tabular}
  \caption{Error table for different $\omega$ values.}
  \label{table:p5}
\end{table}

%% file: Appendices/LocalAverageKernel.tex
\section{Local Average Kernel} \label{appendix:kernel}

A function $K : [-1,1]\to\R$ belongs to the class of functions $\mathbb{K}^{p,q}$ if it satisfies the moment conditions
\begin{subequations} \label{eq:moment-conditions}
  \begin{gather}
    \int_{-1}^1 K(x) \, dx = 1, \\
    \int_{-1}^1 K(x) x^k \, dx = 0, \quad k = 1, \dots, p,
  \end{gather}
\end{subequations}
as well as the compactness conditions
\begin{subequations} \label{eq:compactness-conditions}
  \begin{gather}
    K(\pm 1) = 0, \\
    K^{(\ell)}(\pm1) = 0, \quad \ell = 1, \dots, q.
  \end{gather}
\end{subequations}
Note that the moment conditions, \eqref{eq:moment-conditions}, are equivalent to the condition that for all polynomials $f$ of degree less than or equal to $p$,
\begin{equation} \label{eq:polynomial_delta}
  \int_{-1}^1 K(x) f(x) \, dx = f(0).
\end{equation}
That is, $K$ behaves like a delta function for degree $p$ polynomials. This form of the moment conditions will prove useful.
For using these kernels in computation, we are interested in finding a function $K \in \mathbb{K}^{p,q}$ which is easy to evaluate, namely, a polynomial. We can express such a kernel as the product:
\[ K(x) = (1 - x^2)^{1+q} \zeta(x), \]
where $\zeta$ is a polynomial of degree $p$. This $K$ is the minimum degree polynomial which satisfies \eqref{eq:moment-conditions} and \eqref{eq:compactness-conditions}. Note that the compactness conditions are automatically satisfied.
To determine $\zeta$, define the weighted inner product:
\[ \langle f, g \rangle = \int_{-1}^1 f(x) g(x) (1 - x^2)^{1+q} \, dx. \]
Recall that the Jacobi polynomials $P^{(1+q,1+q)}_n$ for $n \in \N$ are a complete orthogonal basis for the Hilbert space defined by this inner product. Write $\zeta$ as the sum of Jacobi polynomials:
\[ \zeta(x) = \sum_{n=0}^p \alpha_n P_n^{(1+q,1+q)}(x). \]
We then require that the moment condition \eqref{eq:polynomial_delta} is satisfied for each $P^{(1+q,1+q)}_n$, that is, for $m=0,1,\dots, p$
\begin{gather*}
  \int_{-1}^1 K(x) P_m^{(1+q,1+q)}(x) \, dx = P_m^{(1+q,1+q)}(0), \\
  \int_{-1}^1 \left\{ P_m^{(1+q,1+q)}(x) (1-x^2)^{1+q} \sum_{n=0}^p \alpha_n P_n^{(1+q,1+q)}(x) \right\} \, dx = P_m^{(1+q,1+q)}(0), \\
  \sum_{n=0}^p \alpha_n \int_{-1}^1 \left\{ P_m^{(1+q,1+q)}(x)P_n^{(1+q,1+q)}(x)(1-x^2)^{1+q} \right\} \, dx = P_m^{(1+q,1+q)}(0), \\
  \sum_{n=0}^p \alpha_n \left\langle P_m^{(1+q,1+q)}, P_n^{(1+q,1+q)} \right\rangle = P_m^{(1+q,1+q)}(0).
\end{gather*}
By orthogonality of the Jacobi polynomials, we have,
\[ \alpha_m \|P_m^{(1+q,1+q)}\|^2 = P_m^{(1+q,1+q)}(0). \]
Here $\|\cdot\|$ is the $w$ weighted norm. Recall that,
\[ \|P_m^{(1+q,1+q)}\|^2 = \frac{2^{3+2q}}{2m + 2q + 3} \, \frac{\Gamma(m+q+2)^2}{m! \cdot \Gamma(m+2q+3)}, \]
and
\[ P_m^{(1+q,1+q)}(0) =
  \begin{cases}
    0, & k \text{ odd}, \\
    \frac{\left(-\frac{1}{4}\right)^{m/2} \Gamma (m+q+2)}{\left(\frac{m}{2}\right)!\cdot \Gamma (\frac{m}{2}+q+2)}, &  k \text{ even}.
\end{cases} \]
This implies that $\alpha_m = 0$ for odd $m$. For even $m$, write $m = 2k$, then
\begin{align*}
  \alpha_{2k} &= \frac{\left(-\frac{1}{4}\right)^{k} \Gamma (2k+q+2)}{k!\cdot \Gamma (k+q+2)}\cdot \frac{(2k)! \Gamma(2k+2q+3)}{\Gamma(2k+q+2)^2} \cdot \frac{4k+2q+3}{2^{3+2q}} \\
  &= \frac{(-1)^k }{\pi}(4k+2q+3) 2^{2k-1} B\left(k+\frac{1}{2},k+q+\frac{3}{2}\right).
\end{align*}
Here $B$ is the beta function.
Thus, we have
\[ K(x) = \zeta(x)(1-x^2)^{1+q} = (1-x^2)^{1+q}\sum_{k=0}^{\lfloor \frac{p}{2} \rfloor} \alpha_{2k} P^{(1+q,1+q)}_{2k}(x). \]

The two terms recurrence for the Jacobi polynomials allows us to evaluate this sum efficiently.